\documentclass[12pt,letterpaper,reqno]{amsart}
\usepackage[T1]{fontenc}
\usepackage{lmodern}
\usepackage{amssymb,amsmath,amsthm,mathrsfs}
\usepackage{enumitem}
\usepackage[margin=1in]{geometry}
\usepackage[expansion=false]{microtype}
\usepackage[dvipsnames]{xcolor}
\usepackage{hyperref}
\usepackage{bookmark}
\usepackage{bm}
\usepackage{doi}
\hypersetup{
	pdftitle={The modified Cartan conjecture},
	pdfauthor={Alexandre Eremenko; Zongben Xu; Teng Zhang},
	pdfsubject={Normal families, the five-function radius, and infinitesimal hyperbolicity},
	pdfkeywords={Cartan conjecture; normal families; Wronskians; holomorphic curves; Kobayashi--Royden pseudometric},
	colorlinks=true,linkcolor=RoyalBlue,citecolor=ForestGreen!65!black,urlcolor=BrickRed}
\newtheorem{thm}{Theorem}[section]
\newtheorem{lem}[thm]{Lemma}
\newtheorem{prop}[thm]{Proposition}
\newtheorem{cor}[thm]{Corollary}

\newtheorem*{cartanthm}{Cartan's theorem}
\newtheorem*{modifiedconjecture}{Modified Cartan conjecture}
\theoremstyle{definition}

\newtheorem{defin}[thm]{Definition}
\numberwithin{equation}{section}
\newcommand{\D}{\mathbb D}
\newcommand{\C}{\mathbb C}
\newcommand{\PP}{\mathbb P}
\newcommand{\OO}{\mathcal O}
\newcommand{\dd}{\,d}
\newcommand{\norm}[2]{\left\lVert #1\right\rVert_{#2}}
\newcommand{\mean}[2]{\left\langle #1\right\rangle_{#2}}
\DeclareMathOperator{\Rea}{Re}
\DeclareMathOperator{\Ima}{Im}
\DeclareMathOperator{\diam}{diam}
\setlist[enumerate]{label=\textup{(\roman*)},leftmargin=2.2em,itemsep=2pt,topsep=4pt}
\begin{document}
	\title{The modified Cartan conjecture}
	\author[A.~Eremenko]{Alexandre Eremenko}
	\address{Mathematics Department, Purdue University,
		West Lafayette, IN 47907, USA}
	\email{eremenko@purdue.edu}
	\author[Z.~Xu]{Zongben Xu}
	\address{School of Mathematics and Statistics, Xi'an Jiaotong University,
		Xi'an 710049, P.~R.~China}
	\email{zbxu@mail.xjtu.edu.cn}
	\author[T.~Zhang]{Teng Zhang}
	\address{School of Mathematics and Statistics, Xi'an Jiaotong University,
		Xi'an 710049, P.~R.~China}
	\email{teng.zhang@stu.xjtu.edu.cn}
	\subjclass[2020]{Primary 32Q45; Secondary 30D35, 32D45}
	\keywords{holomorphic curve, Bloch's Principle, Nevanlinna theory, Wronskian,
         Kobayashi--Royden pseudometric, Cartan's 1928 conjecture}
	\date{September 14, 2026}
	\begin{abstract}
	In 1928 H. Cartan stated a conjecture about holomorphic curves in
         $\PP^n$
		parametrized by the unit disc, omitting $n+2$ hyperplanes in general position.
		He proved it for $n=2$. In 1996 the first-named author constructed
		counterexamples for all $n\geq 3$, and proposed a modified form of Cartan's
		conjecture which he proved for $n=3$.
		In this paper we prove this modified conjecture in all dimensions.
		As a byproduct we prove a quantitative version of the classical theorem that
		analytic functions are linearly dependent iff their Wronskian determinant is
		zero. We interpret our results in
		terms of the Kobayashi metric on the complement of $n+2$
hyperplanes, and on smooth algebraic varieties in the algebraic torus. We also provide a Lean 4 formalization of the main results.
	\end{abstract}
	\maketitle
	
	\section{Introduction}
	
	Emile Borel \cite{Bor97} proved that the identity
	$$f_1+\ldots+f_p=0$$
	for zero-free entire functions can hold only for trivial reasons:
	when the set $\{ f_1,\ldots,f_p\}$ can be partitioned into disjoint groups
	such that within each group all functions are constant multiples
	of one of them,
	and the sum of all functions in each group is zero.
	It follows from the last property that
	each group contains at least two elements.
	
	The case $p=3$ is equivalent to Picard's theorem \cite{Pic79}. Indeed, if
	$f_1,f_2,f_3$ are zero-free entire functions with zero sum, then
	$f=-f_1/f_3$ is entire and omits $0$ and $1$. Conversely, every
	entire function $f$ omitting $0$ and $1$ gives the zero-free triple
	$(-f,f-1,1)$ with zero sum. Each group in Borel's theorem contains
	at least two elements, so there is only one group when $p=3$.
	Thus all $f_j$ are proportional to one of them, so $f$ is constant.
	
	So Borel's theorem is a generalization of Picard's theorem to
	holomorphic curves in projective space, parametrized by the complex plane.
	
	Andre Bloch was promoting the philosophical principle that ``Nihil est
	in infinito quod non prius fuerit in finito''. When applied to function
	theory this is usually interpreted as {\em every condition which implies
		that a meromorphic function in the plane is constant, when imposed on
		a family of meromorphic functions in an arbitrary region, must imply
		normality (pre-compactness) of this family.} The prototype is Montel's
	normality criterion \cite{Mon27} which says that a family of holomorphic functions
	omitting $0$ and $1$ is normal in the spherical metric.
	For a general discussion of Bloch's Principle
	we refer to the original exposition \cite{Blo26a} and to
	the modern survey  \cite{Ber06}. For extensions to holomorphic curves we mention
	\cite{Yam26}.
	
	Bloch \cite{Blo26} applied his principle to Borel's theorem, and obtained some
	partial results, some of them only heuristic, without complete proofs.
	Then this question was studied by
	Henri Cartan in his thesis \cite{Car28}. To state Cartan's result we use the
	following notation and definition.
	
	We write $D(r)=\{z\in\C:|z|<r\}$, $\D=D(1)$, and
	$\C^*=\C\setminus\{0\}$.
	By $\OO(U)$ we denote the ring of holomorphic functions on an open set $U$
	with the topology of uniform convergence on compact subsets.
	A holomorphic function without zeros will be called a \emph{unit},
	in accordance with the terminology of ring theory.
	For $p\geq 3$, we consider sequences of $p$-tuples of
	units $(f_{1,n},\ldots,f_{p,n})_{n=1}^\infty$
	on $\D$ such that
	\begin{equation}\label{eq:zero-sum}
		f_{1,n}+\cdots+f_{p,n}=0,\quad n=1,2,\ldots
	\end{equation}
	The appropriate counterpart for a group of proportional
	functions is the following
	notion, introduced in \cite{Ere96b}.
	
	\begin{defin}\label{def:cclass}
		{\em Let $\Omega\subset\D$ be a region. A subset
		$I\subseteq\{1,\ldots,p\}$ is called a {\em $C$-class}
on $\Omega$ if there is an
		index $k\in I$ such that the quotients $f_{j,n}/f_{k,n}$, $j\in I$,
		are locally uniformly bounded on $\Omega$ and
		\begin{equation}\label{cclass}
			\sum_{j\in I}\frac{f_{j,n}}{f_{k,n}}\longrightarrow0
			\quad\text{locally uniformly on }\Omega,\quad\mbox{as}\quad n\to\infty.
		\end{equation}
		Such an index $k$ is called a \emph{dominant index} of the $C$-class.}
	\end{defin}
	
	The dominant index is fixed (independent of $n$).
	Every $C$-class has cardinality at least two, since for a singleton
	$I=\{k\}$ the normalized sum in \eqref{cclass} is identically equal to $1$.

	In \cite[pp.~312--318]{Car28},
	Cartan proved:
	
	\begin{cartanthm}
		Under the assumption \eqref{eq:zero-sum}, after selection of a subsequence,
		either all indices form a single $C$-class or there are at least two disjoint
		$C$-classes on $\D$.
	\end{cartanthm}
	
	In particular, a partition into $C$-classes
	exists on the whole disk for $p\leq 4$.
	His conjecture that such a partition exists
	for every $p\geq 3$ was
	disproved in \cite{Ere96a}, where  counterexamples were constructed
	for all $p\ge5$. Then the following
	modification was proposed in
	\cite{Ere96b,Ere15}:
	
	\begin{modifiedconjecture}
		After passage to a subsequence, every system \eqref{eq:zero-sum}
		admits a partition into $C$-classes on a concentric disk of positive
		radius depending only on $p$.
	\end{modifiedconjecture}
	
	This assertion was proved for $p=5$ 
	in \cite{Ere96b}.
	More recently, Yamanoi \cite[Theorem~1.7]{Yam26} obtained, after passage
	to a subsequence, disjoint nonempty $C$-classes $I_1,\ldots,I_q$ on the
	whole disk, together with control of the remaining indices away from
	exceptional sets of arbitrarily small total radius. More precisely, set
$
		I=I_1\sqcup\cdots\sqcup I_q.
$
	Then, for every $j\notin I$, every $0<s<1$, every $\gamma>0$, and every
	$\varepsilon>0$, one has, for all sufficiently large $n$,
	\[
		\frac{|f_{j,n}(z)|}
		{\left(\sum_{k\in I}|f_{k,n}(z)|^2\right)^{1/2}}
		<\varepsilon
	\]
	for $\gamma$-almost all $z\in D(s)$, where ``$\gamma$-almost all'' means
	outside a set covered by at most countably many closed disks whose sum of
	radii is less than $\gamma$.

 This result does not imply that every index belongs to one of
	the $C$-classes. Our main theorem establishes the full partition of $\{1,\ldots,p\}$ into $C$-classes after a
	restriction to a fixed disk, and proves the Modified Conjecture.
	
	\begin{thm}[Partition theorem]\label{thm:main}
		For every integer $p\ge3$ there is $\varepsilon_p\in(0,1)$ such that
		every sequence of $p$-tuples of units on $\D$ satisfying
		\eqref{eq:zero-sum} has a subsequence for which $\{1,\ldots,p\}$ is a
		disjoint union of $C$-classes on $D(\varepsilon_p)$.
	\end{thm}
	
	The statement is evidently conformally invariant:
	it holds not only for the disk $D(\varepsilon_p)$ but for every hyperbolic disk
	of the same hyperbolic radius.
	
	Certainly the whole unit disk can be
	covered by hyperbolic disks of fixed radius
	but this does not imply that the conclusion of the theorem
	holds in the whole unit disk:
	the reason is that a partition into $C$-classes for a given sequence of
	$p$-tuples is not unique,
	so one can have different partitions on overlapping disks. 
	
	For $p=5$ we obtain the optimal value of $\varepsilon_p$, and can replace
	disks by arbitrary regions of fixed hyperbolic diameter.
	We use the hyperbolic length element
	$2|dz|/(1-|z|^2)$, denote the hyperbolic distance by $d$,
	and define $\diam_\D\Omega=\sup_{z,w\in\Omega}d(z,w)$.
	
	\begin{thm}\label{thm:sharp-five}
		Let $\Omega\subset\D$ be a non-empty open set
		with $\diam_{\D}\Omega\leq \log3$.
		Every sequence of five-tuples of
		units on $\D$ whose sum is zero has a subsequence admitting a
		partition into $C$-classes on $\Omega$.
	\end{thm}
	
	The constant $\log3$ cannot be increased. Let $R_p$ denote the
	least upper bound of radii $\varepsilon$ of Euclidean disks centered
	at the origin for which
	the conclusion of Theorem~\ref{thm:main} holds with $\varepsilon_p=\varepsilon$. Then
	\begin{equation*}
		R_5=2-\sqrt3.
	\end{equation*}
	Indeed, the hyperbolic diameter of $D(R)$ is $4\operatorname{arctanh}R$;
	when $R=2-\sqrt{3}$ we obtain the hyperbolic diameter $\log 3$.
	
	
	A $p$-tuple of units satisfying \eqref{eq:zero-sum} defines a
	holomorphic curve in the hyperplane $\sum_{j=1}^p x_j=0$ of
	$\PP^{p-1}$, which is isomorphic to $\PP^{p-2}$.
	The curve omits the $p$ coordinate hyperplanes restricted to this
	hyperplane; these are in general position.
	
	The connection between geometry of hyperplane complements
	and holomorphic curves in them
	was studied by many authors
	beginning with Kiernan and
	Kobayashi \cite{KK73}; Lang \cite[Chapter~VIII]{Lan87} gives a
	systematic account. Our Theorem~\ref{thm:main} also has an application
	to the Kobayashi--Royden pseudometric on
	arbitrary smooth closed algebraic subvarieties of the algebraic torus
	$(\C^*)^N$ (that is, a product of copies of $\C^*$). For a complex
	manifold $Y$, we use the normalization
	\[
	\kappa_Y(x;v)=\inf\{t^{-1}:t>0,\ F\in\OO(\D,Y),\
	F(0)=x,\ F'(0)=tv\}
	\]
	of the Kobayashi--Royden pseudometric \cite{Roy71}.
	Here $\OO(\D,Y)$ denotes the set of holomorphic maps from $\D$ to $Y$.
	Yamanoi \cite[Theorem~1.5]{Yam26} proved that a zero direction
	for this pseudometric is realized by an entire curve when $Y$ is a smooth
	closed subvariety of an abelian variety. He
	\cite[p.~459 and Remark~14.4 on p.~604]{Yam26} notes that it is unknown
	whether the same conclusion holds for noncompact semi-abelian varieties. The following theorem
answers Yamanoi's question for smooth closed  algebraic subvarieties of the
algebraic
torus and specifies the entire curve.
	
	For $x=(x_1,\ldots,x_N)\in(\C^*)^N$ and
	$v=(v_1,\ldots,v_N)\in T_x(\C^*)^N$, put
	\begin{equation*}
		\lambda_j=\frac{v_j}{x_j},\qquad
		\gamma_{x,v}(\zeta)=
		(x_1e^{\lambda_1\zeta},\ldots,x_Ne^{\lambda_N\zeta}).
	\end{equation*}
	
	\begin{thm}\label{thm:torus-zero}
		Let $Y\subset(\C^*)^N$ be a smooth closed algebraic subvariety.
		For $x\in Y$ and $v\in T_xY$,
		\[
		\kappa_Y(x;v)=0
		\quad\Longleftrightarrow\quad
		\gamma_{x,v}(\C)\subset Y.
		\]
	\end{thm}
	
	The zero directions form a closed algebraic subset of $TY$, and
	$\kappa_Y$ has a positive lower bound on every compact subset of its
	complement. These additional properties are proved in
	Section~\ref{sec:geometry}. For the complement of $p$ hyperplanes in
	$\mathbb P^{p-2}$ associated with \eqref{eq:zero-sum}, the theorem
	gives an explicit description in terms of
independently rescaled
	groups of coordinates whose sums are zero.
	
	The main analytic step in the proof of Theorem~\ref{thm:main} is the following estimate
	(Proposition~\ref{thm:absorption}).
	\vspace{.1in}

	{\em Let $A_{i,n}$ be units and $a_{i,n}$ holomorphic functions on
		$\D$, where $1\le i\le m$. Suppose that
		\[
		\frac{a_{i,n}}{A_{i,n}}\longrightarrow0\quad(1\le i\le m),
		\qquad s_n:=\sum_{i=1}^m a_{i,n}\longrightarrow s\not\equiv0
		\]
		locally uniformly on $\D$. Then there exist an index $i$ and a
		subsequence such that $1/A_{i,n}\longrightarrow0$ locally uniformly on $D(r_m)$,
		where $r_m>0$ depends only on $m$.}
	\vspace{.1in}

	The functions $a_{i,n}$, their sums $s_n$, and the limit $s$
	may have zeros. We prove the estimate by induction, using a
	quantitative estimate for Wronskians and logarithmic derivatives.
	
	The quantitative estimate of the Wronskian is another new ingredient which
	may have independent interest
	(Proposition~\ref{prop:wronskian}).
	It shows stability of the classical
	theorem that holomorphic functions are linearly dependent iff their Wronskian
	is identically equal to zero.
	We show that if the Wronskian is small, and the functions holomorphic
	in $\D$ are properly
	normalized, then the functions are almost linearly dependent,
	in the sense that there is a linear combination, whose coefficients
	have unit $\ell^2$-norm, which is small (in a smaller disk).
	
	To derive this, we use the tools employed by Cartan:
	the Bloch--Cartan estimate
	(Lemma~\ref{lem:cartan-circle}), the lemma on the logarithmic derivative
	(Lemma~\ref{lem:logderivative}), the Borel--Nevanlinna growth lemma
	(Lemma~\ref{lem:growth}), Harnack's inequality
	(Lemma~\ref{lem:envelope}), and Jensen's inequality
	(Lemma~\ref{lem:poisson-mean}).
	
	We mention that there is an alternative method
of proving Proposition~\ref{prop:wronskian}
based on the recent Universal
Pl\"ucker-coordinate theorem of Karp and Purbhoo~\cite{KP26}, but we choose
an exposition in the classical spirit. 
	
	Lang notes in \cite[Chapter~VIII, p.~225]{Lan87} that he was unable
	to simplify Cartan's arguments. Theorem~\ref{thm:main} does not imply
	Cartan's result, because its conclusion only applies to a smaller disk.
	
In Section~\ref{sec:wronskian} we establish the analytic estimates.
	In Section~\ref{sec:absorption} we prove Proposition~\ref{thm:absorption}, and
	in Section~\ref{sec:partition} deduce Theorem~\ref{thm:main}.
	Section~\ref{sec:sharp-five} contains the sharp five-function result, and
	Section~\ref{sec:geometry} treats the geometric applications.
	
	Throughout the paper,
	\[
	\norm{h}{r}=\max_{|z|\le r}|h(z)|,\qquad
	\mean{u}{r}=\frac1{2\pi}\int_0^{2\pi}u(re^{i\theta})\dd\theta,
	\qquad m(r,h)=\mean{\log^+|h|}{r},
	\]
	where $\log^+t=\max\{0,\log t\}$ and
	$\log^-t=\max\{0,-\log t\}$.
	
	We use $\|f\|_{L^p}$ for the $L^p$-norm of a scalar-valued function
	$f$, with the domain and measure specified when used. For $\bm{a}=(a_1,\ldots,a_n)\in\C^n$, we use
	$\|\bm{a}\|_{\ell^p}$ to denote its $\ell^p$-norm; in particular,
	$\|\bm{a}\|_{\ell^2}$ is its Euclidean norm.
	
	All limits of holomorphic functions are
	locally uniform unless stated otherwise. Subsequences are relabelled
	by $n$. Constants denoted by $C$ may change from line to line, with
	their fixed parameters indicated when needed.

    A Lean~4 formalization of the main results and supporting lemmas
of this paper, including Theorems~\ref{thm:main},
\ref{thm:sharp-five}, and~\ref{thm:torus-zero},
the sharp identity $R_5=2-\sqrt{3}$, and the geometric applications,
is provided in~\cite{Zha26}.

  \medskip
\noindent\textbf{Acknowledgments and AI tools disclosure. } Teng Zhang is supported by the China Scholarship Council, the Young Elite Scientists Sponsorship Program for PhD Students of the China Association for Science and Technology, and the Fundamental Research Funds for the Central Universities at Xi'an Jiaotong University (Grant No.~xzy022024045).

ChatGPT was used for English-language editing, proofreading, grammatical corrections, and as an exploratory tool for discussing possible approaches to selected parts of this paper. The mathematical arguments and proofs in the original manuscript were independently developed, checked, and written by the authors. 

The Lean~4 formalization in~\cite{Zha26} was generated using
OpenAI's Codex and checked by the Lean kernel.
	
	\section{Quantitative analytic estimates}\label{sec:wronskian}
	
	\medskip\noindent\textbf{Wronskians.}
	
	For holomorphic functions $g_1,\ldots,g_m$, write
	\[
	W(g_1,\ldots,g_m)=\det(g_j^{(i-1)})_{i,j=1}^m.
	\]
	The Wronskian vanishes identically precisely when the functions are
	linearly dependent over the constants; see \cite[Theorem~1]{BD10}.
	The main result of this section is a quantitative estimate of the
	Wronskian from below.
	
	We first state a lower estimate of
	Bloch--Cartan type; compare Cartan's lemma in
	\cite[Th\'eor\`eme~III, pp.~273--276]{Car28}, or \cite[Chap. VIII, \S3]{Lan87}.
	The proof below gives the precise power dependence needed here.
	
	\begin{lem}\label{lem:cartan-circle}
		Fix $0<a<b<c<1$. There is $\gamma>0$, depending only on $a,b,c$,
		such that every $F\in\OO(\D)$ with $|F|\le1$ and
		$\norm{F}{a}\ge t$, where $0<t\le1$, admits a radius $\rho\in(b,c)$
		for which
		\begin{equation}\label{eq:cartan-circle}
			\min_{|z|=\rho}|F(z)|\ge t^\gamma.
		\end{equation}
	\end{lem}
	
	\begin{proof}
		If $t=1$, the maximum principle makes $F$ a constant of modulus one.
		Suppose that $t<1$. Choose fixed radii
		$c<T_-<T_+<S<1$, and choose $T\in(T_-,T_+)$ so that $F$ has no
		zero on $|z|=T$. List its zeros in $D(T)$, with multiplicities, as
		$a_1,\ldots,a_N$. For $|\zeta|<S$ put
		\[
		b_{\zeta,S}(z)=\frac{S(z-\zeta)}{S^2-\overline\zeta z}.
		\]
		Choose $w\in\overline{D(a)}$ with $|F(w)|\ge t$. There is a fixed
		$q<1$ such that $|b_{\zeta,S}(w)|\le q$ whenever $|\zeta|\le T_+$.
		Dividing $F$ by the corresponding $N$ factors and applying the
		maximum principle on $D(S)$ gives $t\le q^N$, hence
		\begin{equation}\label{a}
			N\le C\log(1/t).
		\end{equation}
		In $D(T)$ factorize $F=B_TQ$, where
		$B_T=\prod_{\nu=1}^N b_{a_\nu,T}$. The function $Q$ is zero-free,
		$|Q|\le1$, and $|Q(w)|\ge t$. Harnack's inequality for
		$-\log|Q|$ gives
		\begin{equation}\label{b}
			\log|Q(z)|\ge-C\log(1/t)\qquad(|z|\le c),
		\end{equation}
		uniformly for $T\in(T_-,T_+)$.
		
		For $|z|=\rho\in(b,c)$,
		\[
		|b_{a_\nu,T}(z)|\ge\frac{|\rho-|a_\nu||}{2},
		\qquad
		\sup_{0\le u\le T_+}\int_b^c\log\frac2{|\rho-u|}\dd\rho<\infty.
		\]
		If $N>0$, averaging the sum of these logarithms in $\rho$ supplies
		a radius in $(b,c)$ for which
		\begin{equation}\label{c}
			\sum_{\nu=1}^N\log\frac2{|\rho-|a_\nu||}\le CN.
		\end{equation}
		If $N=0$, any radius in $(b,c)$ suffices. Combining (\ref{a}), (\ref{b})
		and (\ref{c}) 
		we prove \eqref{eq:cartan-circle}.
	\end{proof}
	
	For $\bm{g}=(g_1,\ldots,g_m)$ and $0<a<1$, put
	\[
	\Lambda_a(\bm{g})
	=\inf\left\{\left\|\sum_{j=1}^m \lambda_jg_j\right\|_a:\bm{\lambda}\in\C^m,
	\|\bm{\lambda}\|_{\ell^2}=1\right\}.
	\]
	The infimum is attained on the compact coefficient sphere.
	Lemma~\ref{lem:cartan-circle} permits an elementary induction on
	the number of functions.
	
	\begin{prop}\label{prop:wronskian}
		For every $m\ge1$ and $0<a<b<1$, there are constants $c>0$ and
		$K\ge m$, depending only on $m,a,b$, such that holomorphic functions
		$g_1,\ldots,g_m$ with $|g_j|\le1$ on $\D$ satisfy
		\begin{equation}\label{eq:wronskian-power}
			\norm{W(g_1,\ldots,g_m)}{b}
			\ge c\,\Lambda_a(g_1,\ldots,g_m)^K.
		\end{equation}
	\end{prop}
	
	\begin{proof}
		Put $\tau=\Lambda_a(\bm{g})$, so $0\le\tau\le1$.
		The case $\tau=0$ is immediate, and for $m=1$ take $c=K=1$.
		Assume the result for $m-1$, and choose
		\[
		a<b_1<b_2<b_3<b.
		\]
		Set $V=W(g_1,\ldots,g_{m-1})$ and
		$\Delta=\norm{W(g_1,\ldots,g_m)}{b}$.
		The infimum in $\Lambda_a$ for the first $m-1$ functions is at least $\tau$.
		The induction hypothesis therefore gives
		$\norm{V}{b_1}\ge c_1\tau^{K_1}$.
		Cauchy's estimates bound $V$ on every fixed smaller disk.
		Normalize $V$ by such a bound on a disk of radius between $b_3$
		and $1$, and rescale that disk to $\D$.
		Lemma~\ref{lem:cartan-circle} then gives $\rho\in(b_2,b_3)$ and
		constants $c_2,\beta>0$, depending only on the fixed data, such that
		\begin{equation}\label{eq:minor-circle-lower}
			\min_{|z|=\rho}|V(z)|\ge c_2\tau^\beta.
		\end{equation}
		
		Put $q=m-1$. On a neighborhood of the circle $|z|=\rho$, we define
		\[
		Y=(g_j^{(i-1)})_{i,j=1}^q,\qquad
		\bm{v}=(g_m,g_m',\ldots,g_m^{(q-1)})^{\mathsf T},
		\qquad \bm{d}=Y^{-1}\bm{v}.
		\]
		These are single-valued holomorphic functions there, and $\det Y=V$.
		Differentiating $Y\bm{d}=\bm{v}$ shows that the first
		$q-1$ entries of $Y\bm{d}'$ vanish; the last is the Schur
		complement in the full Wronskian. Thus, with $\bm{e}_q$
		the last coordinate vector,
		\begin{equation}\label{eq:wronskian-coefficient-derivative}
			\bm{d}'
			=Y^{-1}\bm{e}_q\,\frac{W(g_1,\ldots,g_m)}{V}.
		\end{equation}
		The entries and cofactors of $Y$ are uniformly bounded by Cauchy's
		estimates. Equations \eqref{eq:minor-circle-lower} and
		\eqref{eq:wronskian-coefficient-derivative} imply
		\[
		\|\bm{d}'(z)\|_{\ell^2}\le C\Delta\tau^{-2\beta}
		\qquad(|z|=\rho).
		\]
		Fix $z_*,\; |z_*|=\rho$, and integrate along circular arcs of length
		at most $2\pi\rho$. The first row of $Y\bm{d}=\bm{v}$
		then shows that the holomorphic function
		\[
		G(z)=g_m(z)-\sum_{j<m}d_j(z_*)g_j(z)
		\]
		satisfies $\max_{|z|=\rho}|G(z)|\le C\Delta\tau^{-2\beta}$.
		Its constant coefficient vector has norm at least one.
		The definition of $\tau$ and the maximum principle give
		\[
		\tau\le\norm{G}{a}\le C\Delta\tau^{-2\beta}.
		\]
		This proves \eqref{eq:wronskian-power} with exponent $2\beta+1$.
		Since $\tau\le1$, increasing the exponent preserves the inequality,
		so we may arrange $K\ge m$.
	\end{proof}
	
	For the rest of the paper, fix $K_1=1$ and, for each $m\ge2$, an
	exponent $K_m\ge m$ furnished by Proposition~\ref{prop:wronskian}
	for $a=1/4$ and $b=1/2$; denote the corresponding constant by $c_m$.
	Only the existence of these finite exponents is needed.
	
	\medskip\noindent\textbf{Logarithmic derivatives.}
	
	We need a localized higher-order form of the classical lemma on the
	logarithmic derivative, with an interior maximum in place of a value
	at the origin; compare \cite[Chapter~VIII]{Lan87}. For sharper
	classical estimates, see Miles \cite{Mil92}. We give the required
	version. 
	
	\begin{lem}\label{lem:logderivative}
		Fix $0<\alpha<r_-<r_+<1$ and an integer $k\ge1$.
		There is a constant $C$, depending only on these parameters, such
		that every $h\in\OO(\D)$ with $\norm{h}{\alpha}\ge\tau>0$ satisfies
		\begin{equation}\label{eq:logderivative-bound}
			m\left(r,\frac{h^{(k)}}h\right)
			\le C\left\{
			\log\left(2+m(R,h)+\log^+\frac1\tau\right)
			+\log\frac1{R-r}\right\}
			\qquad(r_-\le r<R\le r_+).
		\end{equation}
	\end{lem}
	
	\begin{proof}
		All constants may depend on the fixed radii and on $k$.
		Put $S=(r+R)/2$ and $d=S-r$. Poisson comparison gives
		\[
		\log|h(z)|\le H\quad(|z|\le S),\qquad
		H=1+\frac{C\,m(R,h)}{R-r}.
		\]
		Write $E=H+\log^+(1/\tau)$, so $E\ge1$. Choose
		$r+d/3<T<r+2d/3$ such that $h$ has no zeros on $|z|=T$.
		Let $a_1,\ldots,a_N$ be the zeros of $h$ in $D(T)$, counted with
		multiplicities, and choose $w\in\overline{D(\alpha)}$ with
		$|h(w)|\ge\tau$.
		
		For $|a|<S$, define the Blaschke factor
		\[
		b_{a,S}(z)=\frac{S(z-a)}{S^2-\overline a z}.
		\]
		It satisfies
		\[
		1-|b_{a,S}(w)|^2
		=\frac{(S^2-|a|^2)(S^2-|w|^2)}{|S^2-\overline a w|^2}.
		\]
		Since $|a|\le T$, $S-T\ge d/3$, and $|w|\le\alpha<r_-$,
		we obtain $|b_{a,S}(w)|\le e^{-cd}$ with a uniform $c>0$.
		Divide $h$ by the product of these factors over its zeros in $D(T)$.
		The quotient is holomorphic on $D(S)$ and bounded by $e^H$ on its
		boundary, so the maximum principle gives
		$\tau\le e^He^{-cdN}$. Thus
		\begin{equation}\label{eq:logderivative-zero-count}
			N\le CE/d.
		\end{equation}
		
		Now factor $h=B_Tq$ in $D(T)$, where
		$B_T=\prod_{\nu=1}^N b_{a_\nu,T}$.
		The function $q$ is zero-free, $|q|\le e^H$, and $|q(w)|\ge\tau$.
		Harnack's inequality for $v=H-\log|q|$ gives $v(0)\le CE$.
		Choose a holomorphic logarithm of $q$ on $D(T)$ and apply the
		Schwarz integral formula to $H-\log q$ on the circle of radius
		$(r+T)/2$. The integral of its real part is $v(0)$, while its
		imaginary constant disappears upon differentiation. We obtain
		\begin{equation}\label{eq:logderivative-unit}
			|(\log q)^{(j)}(z)|\le C_jE/d^{j+1}
			\qquad(|z|=r,\ 1\le j\le k).
		\end{equation}
		
		Define the meromorphic functions $L_1(h)=h'/h$ and
		$L_{j+1}(h)=L_j(h)'$.
		Differentiating the Blaschke factors and using
		\eqref{eq:logderivative-zero-count}--\eqref{eq:logderivative-unit}
		gives
		\begin{equation}\label{eq:logderivative-poles}
			|L_j(h)(z)|\le\frac{C_jE}{d^{j+1}}
			+C_j\sum_{\nu=1}^N|z-a_\nu|^{-j}
			\qquad(|z|=r).
		\end{equation}
		Indeed, the reflected poles are at distance at least $T-r$ from
		this circle and contribute at most $C_jN/d^j$.
		The estimate is needed only away from any poles on the circle.
		
		Set $\gamma=1/(2k)$. For $1\le j\le k$,
		\begin{equation}\label{eq:angular-integrability}
			\sup_{a\in\C}\frac1{2\pi}\int_0^{2\pi}
			|re^{i\theta}-a|^{-j\gamma}\dd\theta\le C_k
			\qquad(r_-\le r\le r_+).
		\end{equation}
		If $|a|\notin[r/2,3r/2]$, the integrand is uniformly bounded.
		Otherwise rotate $a$ to the positive real axis and use
		\[
		|re^{i\theta}-a|^2
		=(r-|a|)^2+2r|a|(1-\cos\theta)\ge c\theta^2
		\qquad(-\pi\le\theta\le\pi),
		\]
		with $j\gamma\le1/2<1$.
		
		Raise \eqref{eq:logderivative-poles} to the power $\gamma$, use
		$(x+y)^\gamma\le x^\gamma+y^\gamma$, and integrate.
		By \eqref{eq:logderivative-zero-count},
		\eqref{eq:angular-integrability}, and concavity of the logarithm,
		\[
		\begin{split}
			m(r,L_j(h))
			&\le\frac1\gamma
			\log\left(1+\mean{|L_j(h)|^\gamma}{r}\right)\\
			&\le C_k\left(1+\log E+\log\frac1d\right).
		\end{split}
		\]
		Repeated differentiation shows that $h^{(k)}/h$ is a polynomial
		with fixed coefficients in $L_1(h),\ldots,L_k(h)$.
		The same bound holds for $m(r,h^{(k)}/h)$.
		Substituting $E$, $H$, and $d$ proves
		\eqref{eq:logderivative-bound}.
	\end{proof}
	
	\medskip\noindent\textbf{Three auxiliary estimates.}
	
	We first state an elementary form of the Borel--Nevanlinna growth
	lemma; compare \cite[Chap. VI, Lemma 3.7]{Lan87}.
	
	\begin{lem}\label{lem:growth}
		Let $M$ be positive, continuous, and nondecreasing on $[a,b]$, where
		$a<b$ and $2/M(a)<b-a$. There is $\rho\in[a,b)$ such that
		\begin{equation}\label{eq:good-growth-radius}
			\rho+\frac1{M(\rho)}<b,\qquad
			M\left(\rho+\frac1{M(\rho)}\right)\le2M(\rho).
		\end{equation}
	\end{lem}
	
	\begin{proof}
		Start with $t_0=a$. If the second inequality in
		\eqref{eq:good-growth-radius} fails at $t_j$, put
		$t_{j+1}=t_j+1/M(t_j)$.
		After $j$ failures, $M(t_j)\ge2^jM(a)$, so all proposed points lie
		below $a+2/M(a)<b$.
		An infinite sequence would converge to a point of $[a,b)$ while
		$M(t_j)$ tends to infinity, contradicting continuity.
		The first stopping point has both required properties.
	\end{proof}
	
	\begin{lem}\label{lem:envelope}
		Suppose that $1/2\le\rho\le3/4$ and $0<\delta\le\eta\le1/64$.
		Let $u_1,\ldots,u_m$ be real harmonic functions on a neighborhood of
		$\overline{D(\rho)}$, with $\max_i u_i>0$ on $|z|=\rho$.
		Let $U$ be the harmonic extension of this boundary maximum and put
		$M=U(0)$. If, for each $i$, there is $z_i\in D(\delta)$ such that
		$u_i(z_i)\le C_0$, where $C_0\ge0$, then
		\begin{align}
			u_i(0)&\le8\delta M+C_0,\label{eq:envelope-origin}\\
			\sup_{|z|\le4\eta}u_i(z)&\le64\eta M+C_0.\label{eq:envelope-inner}
		\end{align}
	\end{lem}
	
	\begin{proof}
		The functions $U$ and $U-u_i$ are nonnegative harmonic functions.
		Put $h=(\rho-\delta)/(\rho+\delta)$.
		Harnack's inequality applied to these two functions gives
		\[
		(U-u_i)(0)\ge h(U-u_i)(z_i)\ge h^2M-C_0.
		\]
		Since $1-h^2\le8\delta$, this proves
		\eqref{eq:envelope-origin}.
		For $|z|\le4\eta$, put
		$h_4=(\rho-4\eta)/(\rho+4\eta)$ and apply Harnack once more:
		\[
		u_i(z)\le(h_4^{-1}-h_4 h^2)M+C_0.
		\]
		The coefficient satisfies
		\[
		h_4^{-1}-h_4 h^2
		\le(h_4^{-1}-h_4)+(1-h^2)
		\le\frac{32\eta}{1-64\eta^2}+8\eta\le64\eta,
		\]
		which proves \eqref{eq:envelope-inner}.
	\end{proof}
	
	The next estimate is an off-center version of Jensen's inequality
	for $\log|H|$. When $w=0$, it reduces to
	$\log|H(0)|\le\mean{\log|H|}{\rho}$; the Poisson kernel bounds below
	quantify the loss when $w\ne0$.
	
	\begin{lem}\label{lem:poisson-mean}
		Let $H$ be holomorphic on a neighborhood of $\overline{D(\rho)}$,
		and let $w\in D(\rho)$ with $H(w)\ne0$.
		Put $t=|w|/\rho$ and $q=(1+t)/(1-t)$. Then
		\begin{equation}\label{eq:poisson-mean}
			\mean{\log|H|}{\rho}
			\ge q\log|H(w)|-(q^2-1)m(\rho,H).
		\end{equation}
	\end{lem}
	
	\begin{proof}
		Let $P$ be the Poisson kernel at $w$, normalized for angular measure
		$d\theta/(2\pi)$, so $q^{-1}\le P\le q$.
		Writing $I_+=m(\rho,H)$ and
		$I_-=\mean{\log^-|H|}{\rho}$, subharmonic comparison gives
		\[
		\log|H(w)|\le qI_+-q^{-1}I_-
		=q^{-1}(I_+-I_-)+(q-q^{-1})I_+.
		\]
		Multiplication by $q$ proves \eqref{eq:poisson-mean}.
		At boundary zeros, apply the comparison to $\log(|H|+\varepsilon)$
		and let $\varepsilon\downarrow0$; the logarithmic singularities are
		integrable.
	\end{proof}
	
	\section{Growth estimate for units}\label{sec:absorption}
	
	With the exponents $K_m$ fixed in Proposition~\ref{prop:wronskian}, put
	\begin{equation}\label{eq:recursive-radii}
		r_1=1,\qquad r_2=2-\sqrt3,\qquad
		\eta_m=\frac1{1024(K_m+m)}\ (m\ge2),\qquad
		r_m=\eta_m r_{m-1}\quad(m\ge3).
	\end{equation}
	The radii $r_m$ are positive and decrease with $m$. It is convenient
	to use the sharp value of $r_m$
	at $m=2$ which is proved in Proposition~\ref{prop:sharp-two-absorption} below;
	its proof uses only the auxiliary estimates of Section 2,
	and is independent of the contents of this Section 3. 
	
	\begin{prop}\label{thm:absorption}
		Let $m\ge1$. For every $n$, let $A_{1,n},\ldots,A_{m,n}$ be units on
		$\D$, and let $a_{1,n},\ldots,a_{m,n}\in\OO(\D)$. Suppose that
		\begin{equation}\label{eq:absorption-hypotheses}
			\frac{a_{i,n}}{A_{i,n}}\longrightarrow0\quad(1\le i\le m),
			\qquad
			s_n:=\sum_{i=1}^{m}a_{i,n}\longrightarrow s\not\equiv0
		\end{equation}
		locally uniformly on $\D$. There exist an index $i$ and a subsequence
		such that $1/A_{i,n}\to0$ locally uniformly on $D(r_m)$, where $r_m$
		is defined by \eqref{eq:recursive-radii}.
	\end{prop}
	\begin{proof}
		We argue by induction on $m$.
		
		For $m=1$, let $K\Subset\D$ be arbitrary. Choose a concentric circle
		containing $K$ and avoiding the zeros of $s$. Since $a_{1,n}=s_n\to s$
		locally uniformly, $a_{1,n}$ is bounded away from zero on this circle
		for all sufficiently large $n$. Hence
		\begin{equation*}
			\frac1{A_{1,n}}=\frac{a_{1,n}/A_{1,n}}{a_{1,n}}\longrightarrow0,\quad
			n\to\infty
		\end{equation*}
		uniformly on the circle. The maximum principle for the holomorphic
		function $1/A_{1,n}$ then gives convergence uniformly on the enclosed
		disk. Since $K$ was arbitrary, the assertion follows for $r_1=1$.
		
		For $m=2$, the conclusion follows from
		Proposition~\ref{prop:sharp-two-absorption}, and therefore in particular
		holds on $D(r_2)$.
		
		Let $m\ge3$, assume the result for $m-1$, and abbreviate
		\begin{equation*}
			\eta=\eta_m,\qquad \delta=r_m=\eta r_{m-1}.
		\end{equation*}
		Assume for contradiction that there is no fixed index $i$ and no
		subsequence for which $1/A_{i,n}\to0$ locally uniformly on
		$D(\delta)$. Every further extraction below preserves this assumption.
		
		\smallskip\noindent\emph{A lower bound for linear combinations.}
		For each $i$, zero is not a subsequential limit of $1/A_{i,n}$ in the
		compact-open topology on $D(\delta)$. Hence there exist a compact set
		$K_i\Subset D(\delta)$ and a number $\varepsilon_i>0$ such that
		\begin{equation*}
			\sup_{K_i}|1/A_{i,n}|\ge\varepsilon_i
		\end{equation*}
		for all sufficiently large $n$. Otherwise, a diagonal selection along
		an exhaustion of $D(\delta)$ by compact disks would produce the
		forbidden convergence to zero. Then we
		may choose a common constant $C_0\ge0$ and points
		$z_{i,n}\in D(\delta)$ such that
		\begin{equation}\label{eq:failure-points}
			u_{i,n}(z_{i,n})\le C_0,
			\qquad u_{i,n}=\log|A_{i,n}|,\quad 1\leq i\leq m.
		\end{equation}
		The functions $u_{i,n}$ are harmonic because the $A_{i,n}$ are units.
		
		Consider the least norm of a linear combination whose coefficient
		vector has Euclidean norm one,
		\begin{equation*}
			\lambda_n=\Lambda_\eta(a_{1,n},\ldots,a_{m,n}).
		\end{equation*}
		We claim that there is a constant $\lambda>0$, independent of $n$, such
		that
		\begin{equation}\label{eq:absolute-gap}
			\lambda_n\ge\lambda
		\end{equation}
		for all sufficiently large $n$.
		
		If not, pass to a subsequence for which $\lambda_n\to0$ and choose
		minimizing unit vectors
		$\bm{c}_n=(c_{1,n},\ldots,c_{m,n})\in\C^m$. Put
		\begin{equation*}
			b_n=\sum_i c_{i,n}a_{i,n}.
		\end{equation*}
		After a further extraction of a subsequence, one fixed index $j$ satisfies
		\begin{equation*}
			|c_{j,n}|=\max_i|c_{i,n}|\ge m^{-1/2}
		\end{equation*}
		for every $n$. Since $\|b_n\|_\eta=\lambda_n\to0$, the functions $b_n$
		tend uniformly to zero on $\overline{D(\eta)}$, and
		\begin{equation*}
			\sum_{i\ne j}\left(1-\frac{c_{i,n}}{c_{j,n}}\right)a_{i,n}
			=s_n-\frac{b_n}{c_{j,n}}\longrightarrow s
		\end{equation*}
		locally uniformly on $D(\eta)$. The coefficients on the left have
		modulus at most two. Retaining the units $A_{i,n}$ for $i\ne j$ and
		rescaling $D(\eta)$ to $\D$, the induction hypothesis applies to these
		$m-1$ terms. It follows that, for some fixed $i\ne j$ and a
		subsequence,
		\begin{equation*}
			\frac1{A_{i,n}}\longrightarrow0
			\qquad\text{locally uniformly on }D(\eta r_{m-1})=D(\delta),
		\end{equation*}
		contrary to our standing assumption. This proves
		\eqref{eq:absolute-gap}. In particular,
		\begin{equation*}
			\norm{a_{i,n}}{\eta}\ge\lambda
			\qquad(1\le i\le m).
		\end{equation*}
		
		\smallskip\noindent\emph{A circle with controlled growth.}
		Discard finitely many terms so that
		$|a_{i,n}|\le|A_{i,n}|$ on $\overline{D(7/8)}$ for every $i$. For
		$0<t<1$, define
		\begin{equation*}
			M_n(t)=\mean{\max\{0,u_{1,n},\ldots,u_{m,n}\}}{t}.
		\end{equation*}
		This function is continuous and nondecreasing in $t$, since it is the
		circular mean of a subharmonic function. Moreover,
		\begin{equation}\label{eq:diverging-growth}
			M_n(1/2)\longrightarrow\infty.
		\end{equation}
		Indeed, choose $z_*\in D(1/2)$ with $s(z_*)\ne0$. From
		\eqref{eq:absorption-hypotheses},
		\begin{equation*}
			|s_n(z_*)|
			\le\sum_i\left|\frac{a_{i,n}}{A_{i,n}}(z_*)\right|
			|A_{i,n}(z_*)|.
		\end{equation*}
		Since $s_n(z_*)\to s(z_*)\ne0$, whereas
		$a_{i,n}/A_{i,n}\to0$, we obtain
		$\max_i|A_{i,n}(z_*)|\to\infty$. Poisson comparison for the
		nonnegative subharmonic function
		$\max\{0,u_{1,n},\ldots,u_{m,n}\}$ then proves
		\eqref{eq:diverging-growth}.
		
		Choose a fixed closed interval $[a,b]\subset(1/2,3/4)$ such that $s$
		has no zeros in the closed annulus $a\le|z|\le b$. Such an interval
		exists because $s\not\equiv0$. Hence $|s_n|$ is bounded below by a
		positive constant on this annulus for all sufficiently large $n$. By
		\eqref{eq:diverging-growth} and Lemma~\ref{lem:growth}, we may choose
		$\rho=\rho_n\in[a,b]$ and put
		\begin{equation}\label{eq:chosen-radii}
			M=M_n(\rho),\qquad
			\widehat\rho=\rho+M^{-1}<b,\qquad
			M_n(\widehat\rho)\le2M.
		\end{equation}
		Here and below $M\to\infty$ as $n\to\infty$.
		
		On $|z|=\rho$, the positive lower bound for $|s_n|$ and the uniform
		smallness of all $a_{i,n}/A_{i,n}$ imply
		$\max_i u_{i,n}>0$ for all sufficiently large $n$. Let $U_n$ be the
		harmonic extension of this boundary maximum. Then
		\begin{equation*}
			U_n>0,\qquad U_n(0)=M,\qquad U_n-u_{i,n}\ge0.
		\end{equation*}
		Applying Lemma~\ref{lem:envelope} with
		\eqref{eq:failure-points}, we obtain
		\begin{equation}\label{eq:small-central-values}
			u_{i,n}(0)\le8\delta M+C_0
		\end{equation}
		and
		\begin{equation}\label{eq:small-disk-growth}
			\norm{a_{i,n}}{4\eta}\le e^{L_n},
			\qquad L_n=64\eta M+C_0.
		\end{equation}
		
		\smallskip\noindent\emph{The Wronskian lower bound.}
		Normalize on the small disk by setting
		\begin{equation*}
			g_{i,n}(z)=e^{-L_n}a_{i,n}(4\eta z),\qquad z\in\D.
		\end{equation*}
		By \eqref{eq:small-disk-growth}, these functions have modulus at most
		one. By \eqref{eq:absolute-gap}, the minimum norm of a linear combination with coefficient vector of Euclidean norm one on $D(1/4)$ is
		at least $e^{-L_n}\lambda$. Put
		$W_n=W(a_{1,n},\ldots,a_{m,n})$. The exact scaling identity is
		\begin{equation*}
			W(g_{1,n},\ldots,g_{m,n})(z)
			=e^{-mL_n}(4\eta)^{m(m-1)/2}W_n(4\eta z).
		\end{equation*}
		Proposition~\ref{prop:wronskian} therefore gives
		\begin{equation*}
			\norm{W_n}{2\eta}
			\ge c_m(4\eta)^{-m(m-1)/2}
			\lambda^{K_m}e^{-(K_m-m)L_n}.
		\end{equation*}
		Choose $w_n\in\overline{D(2\eta)}$ at which $|W_n|$ attains this
		norm. Since $K_m\ge m$ and $L_n=64\eta M+C_0$, we obtain
		\begin{equation}\label{eq:wronskian-point-lower}
			\log|W_n(w_n)|\ge-64K_m\eta M-O(1).
		\end{equation}
		The $O(1)$ term may depend on the fixed data of the sequence, through
		$\lambda$ and $C_0$, but it is independent of $n$. In particular,
		$W_n\not\equiv0$.
		
		\smallskip\noindent\emph{The boundary estimate.}
		Apply Lemma~\ref{lem:logderivative} with interior radius $\eta$, fixed
		outer range $r_-=1/2$, $r_+=7/8$, lower bound $\lambda$, and the radii
		$\rho,\widehat\rho$ from \eqref{eq:chosen-radii}. Since
		$m(\widehat\rho,a_{i,n})\le M_n(\widehat\rho)\le2M$ and
		$\widehat\rho-\rho=M^{-1}$, we have
		\begin{equation}\label{eq:negligible-derivatives}
			m\left(\rho,\frac{a_{i,n}^{(k)}}{a_{i,n}}\right)
			=O(\log M)=o(M)
			\qquad(1\le i\le m,\ 1\le k\le m-1).
		\end{equation}
		
		For each $j$, multilinearity and $s_n=\sum_i a_{i,n}$ give
		\begin{equation*}
			W_n=\pm W(s_n,a_{1,n},\ldots,\widehat{a_{j,n}},\ldots,a_{m,n}),
		\end{equation*}
		where the hat denotes omission. Since $s_n\to s$ locally uniformly,
		the derivatives of $s_n$ through order $m-1$ are uniformly bounded on
		$\overline{D(3/4)}$. Factoring the remaining columns and expanding the
		determinant, we obtain a nonnegative function $E_n$ on $|z|=\rho$,
		independent of $j$, such that, almost everywhere,
		\begin{equation}\label{eq:wronskian-pointwise-upper}
			\log|W_n|\le\sum_{i\ne j}u_{i,n}+E_n,
			\qquad \mean{E_n}{\rho}=o(M).
		\end{equation}
		For instance, one may take a sufficiently large fixed constant plus
		\begin{equation*}
			\sum_i\sum_{k=1}^{m-1}
			\log^+\left|\frac{a_{i,n}^{(k)}}{a_{i,n}}\right|,
		\end{equation*}
		in view of \eqref{eq:negligible-derivatives}.
		
		At each boundary point choose $j$ so that $u_{j,n}$ is maximal, and
		integrate. Harmonicity gives
		\begin{equation}\label{eq:wronskian-mean-upper}
			I_n:=\mean{\log|W_n|}{\rho}
			\le\sum_i u_{i,n}(0)-M+o(M).
		\end{equation}
		For the positive part, put $v_n=\max_i u_{i,n}$ on $|z|=\rho$. Since
		$v_n>0$ and $E_n\ge0$, \eqref{eq:wronskian-pointwise-upper} implies
		\begin{equation}\label{eq:wronskian-positive-pointwise}
			\log^+|W_n|\le(m-1)v_n+E_n.
		\end{equation}
		almost everywhere on this circle. Averaging gives
		\begin{equation}\label{eq:wronskian-positive-mean}
			m(\rho,W_n)\le(m-1)M+o(M).
		\end{equation}
		The logarithmic singularities in \eqref{eq:wronskian-pointwise-upper}
		and \eqref{eq:wronskian-positive-pointwise} are integrable, so the means
		in \eqref{eq:wronskian-mean-upper} and
		\eqref{eq:wronskian-positive-mean} are well-defined.
		
		Apply Lemma~\ref{lem:poisson-mean} at $w_n$. With
		$t=|w_n|/\rho\le4\eta$ and $q=(1+t)/(1-t)$, the bound
		$\eta\le1/64$ implies
		\begin{equation*}
			q\le2,\qquad q^2-1\le20\eta.
		\end{equation*}
		Combining \eqref{eq:wronskian-point-lower},
		\eqref{eq:wronskian-positive-mean}, and
		\eqref{eq:poisson-mean}, we obtain
		\begin{equation*}
			I_n\ge-(128K_m+20m)\eta M-o(M).
		\end{equation*}
		On the other hand, \eqref{eq:small-central-values} and
		\eqref{eq:wronskian-mean-upper} give
		\begin{equation*}
			I_n\le(8m\delta-1)M+o(M).
		\end{equation*}
		Comparing the last two inequalities and using
		\eqref{eq:recursive-radii}, we obtain
		\[
		1\le 8m\delta+(128K_m+20m)\eta+o(1)
		\le\frac{128K_m+28m}{1024(K_m+m)}+o(1)
		\le\frac18+o(1),
		\]
		since $\delta\le\eta$. This is impossible as $n\to\infty$ and completes
		the induction.
	\end{proof}

	\medskip\noindent\textbf{Reduction under linear dependence.}
	
	Linear relations among the functions $a_{i,n}$ in
	Proposition~\ref{thm:absorption} allow a slight improvement of its conclusion:
	
	\begin{cor}\label{cor:rank-adaptive-absorption}
		Under the hypotheses of Proposition~\ref{thm:absorption}, suppose
		in addition that
		\[
		\dim_{\C}\operatorname{span}\{a_{1,n},\ldots,a_{m,n}\}\le d
		\quad\text{for every }n,
		\]
		where $d$ is an integer with $1\le d\le m$. Then some reciprocal $1/A_{i,n}$ tends to zero,
		along a subsequence, locally uniformly on $D(r_d)$, where $r_d$ is defined
		by \eqref{eq:recursive-radii}.
	\end{cor}
	
\begin{proof}
    Since $s\not\equiv0$, the span
    $V_n=\operatorname{span}_{\C}\{a_{1,n},\ldots,a_{m,n}\}$ is nonzero
    for all sufficiently large $n$. After passage to a subsequence,
    its dimension is a fixed integer $q$, where $1\le q\le d$.
    Choose any linear isomorphism $T_n:V_n\to\C^q$.
    Among all $q$-element subsets $I_n\subset\{1,\ldots,m\}$, choose
    one maximizing
$
\left|\det\bigl(T_n a_{i,n}\bigr)_{i\in I_n}\right|,
$
    with the columns in increasing order of their indices. This maximum
    is nonzero, so the functions $a_{i,n}$, $i\in I_n$, form a basis.
    Write
    \[
        a_{j,n}=\sum_{i\in I_n}c_{ij,n}a_{i,n}.
    \]
    Replacing the column indexed by $i$ by $T_n a_{j,n}$ multiplies
    the determinant by $c_{ij,n}$. By maximality of the chosen
    determinant, $|c_{ij,n}|\le1$ for every $i\in I_n$ and every $j$.
    There are finitely many possible sets $I_n$, so after another
    extraction we may assume that $I_n=I$ is fixed. Put
    \[
        b_{i,n}=\left(\sum_{j=1}^m c_{ij,n}\right)a_{i,n}
        \qquad(i\in I).
    \]
    The coefficients in parentheses have modulus at most $m$. Hence
    \[
        \sum_{i\in I}b_{i,n}=s_n,\qquad
        \frac{b_{i,n}}{A_{i,n}}\longrightarrow0
    \]
    locally uniformly on $\D$. Apply Proposition~\ref{thm:absorption}
    to these $q$ terms and use $r_q\ge r_d$.
\end{proof}
	
	\section{Proof of Theorem~\ref{thm:main}}\label{sec:partition}
	
We derive Theorem~\ref{thm:main} from Proposition~\ref{thm:absorption}
by comparing quotients on finitely many concentric disks.

A \emph{preorder} is a reflexive and transitive relation $\preceq$.
It defines an equivalence relation by $i\sim j$ if both $i\preceq j$
and $j\preceq i$. An element $i$ is \emph{maximal} if
$i\preceq j$ implies $j\preceq i$ for every $j$; an equivalence class
is maximal if its elements are maximal. Two elements are
\emph{incomparable} if neither precedes the other.
\begin{lem}\label{lem:stabilization}
    Let $0<\sigma<1$, put $\rho_k=\sigma^k$ for $0\le k\le p-1$,
    and let $f_{1,n},\ldots,f_{p,n}$ be units on $\D$.
    After passage to a subsequence, for every $i,j,k$ the quotient
    $f_{i,n}/f_{j,n}$ either converges locally uniformly on $D(\rho_k)$
    to a holomorphic function, or satisfies
    \[
        \max_K|f_{i,n}/f_{j,n}|\longrightarrow\infty
    \]
    for some fixed compact set $K\Subset D(\rho_k)$.
    For this subsequence define
    \[
        i\preceq_k j\quad\Longleftrightarrow\quad
        (f_{i,n}/f_{j,n})_n\text{ is locally uniformly bounded on }D(\rho_k).
    \]
    We say that $j$ \emph{dominates} $i$ on $D(\rho_k)$ precisely when
    $i\preceq_k j$. This relation is a preorder. Write $i\sim_k j$
    when both $i\preceq_k j$ and $j\preceq_k i$.
    Then there is $\ell\in\{0,\ldots,p-1\}$ such that either
    $\preceq_\ell$ has just one maximal equivalence class, or
    $\ell<p-1$ and indices chosen from distinct maximal equivalence
    classes for $\preceq_\ell$ are incomparable for $\preceq_{\ell+1}$.
\end{lem}

\begin{proof}
    For each of the finitely many pairs $(i,j)$ and integers $k$,
    consider the current sequence of quotients on $D(\rho_k)$.
   If it is locally uniformly bounded, Montel's theorem
\cite[Ch.~V, Theorem~15]{Ahl79}
gives a subsequence converging locally uniformly to a holomorphic function.
    If it is not locally uniformly bounded, there is a fixed compact
    subset on which the suprema are unbounded; choose a subsequence
    along which these suprema tend to infinity. Both conclusions
    persist under further extraction, so finitely many selections
    give all the asserted alternatives simultaneously.

    Reflexivity of $\preceq_k$ follows from $f_{i,n}/f_{i,n}=1$,
    and transitivity follows by multiplying quotients. Restriction
    to a smaller disk gives
    \[
        i\preceq_k j\quad\Longrightarrow\quad i\preceq_{k+1}j.
    \]
    Let $t_k$ be the number of maximal equivalence classes for
    $\preceq_k$. Every maximal class for $\preceq_{k+1}$ contains a
    maximal class for $\preceq_k$. Indeed, given an index $x$ in the
    former, finiteness of the set of indices gives a maximal index
    $y$ for $\preceq_k$ with $x\preceq_k y$. Then
    $x\preceq_{k+1}y$, and maximality of $x$ for $\preceq_{k+1}$
    gives $y\preceq_{k+1}x$. The whole $\sim_k$-class of $y$ is
    therefore contained in the $\sim_{k+1}$-class of $x$.

    Distinct maximal classes for $\preceq_{k+1}$ are disjoint, so
    they contain distinct maximal classes for $\preceq_k$.
    Consequently $t_{k+1}\le t_k$. If $t_{k+1}=t_k$, each maximal
    class for $\preceq_{k+1}$ contains exactly one maximal class
    for $\preceq_k$, and every maximal class for $\preceq_k$ is
    contained in one of them. Thus indices belonging to two distinct
    old maximal classes lie in two distinct new maximal classes. Such
    indices are incomparable for $\preceq_{k+1}$:
    a relation in either direction would, by maximality, make the
    two new classes equal.

    If some $t_k=1$, take $\ell=k$. Otherwise the $p$ nonincreasing
    integers $t_0,\ldots,t_{p-1}$ all belong to $\{2,\ldots,p\}$.
    Hence $t_\ell=t_{\ell+1}$ for some $\ell<p-1$, which gives
    the second conclusion.
\end{proof}

We now prove Theorem~\ref{thm:main}.

\begin{proof}[Proof of Theorem~\ref{thm:main}]
    Set
    \[
        \sigma=r_{p-1},\qquad \varepsilon_p=\sigma^{p-1},
        \qquad \rho_k=\sigma^k\quad(0\le k\le p-1),
    \]
    and apply Lemma~\ref{lem:stabilization}. For the integer $\ell$
    given by the lemma, choose one index $k_\nu$ from each maximal
    equivalence class for $\preceq_\ell$, with $1\le\nu\le t$.
    Every index $j$ satisfies $j\preceq_\ell k_\nu$ for at least
    one $\nu$. Assign $j$ to one such $\nu$, and let $J_\nu$ be
    the set of indices assigned to $\nu$. If an index belongs to a
    maximal equivalence class, assign it to the same set as the chosen
    index $k_\nu$ from that class. Then
    \[
        \{1,\ldots,p\}=J_1\sqcup\cdots\sqcup J_t,
        \qquad k_\nu\in J_\nu,
        \qquad j\preceq_\ell k_\nu\quad(j\in J_\nu).
    \]
    The sets $J_\nu$ are the parts of this partition; they need not
    be equivalence classes for $\sim_\ell$.
    Put $F_{\nu,n}=f_{k_\nu,n}$. By the quotient convergence in
    Lemma~\ref{lem:stabilization},
    \[
        h_{\nu,n}:=\sum_{j\in J_\nu}\frac{f_{j,n}}{F_{\nu,n}}
        \longrightarrow h_\nu
    \]
    locally uniformly on $D(\rho_\ell)$, where $h_\nu$ is holomorphic.
    Moreover,
    \begin{equation}\label{eq:partition-identity}
        \sum_{\nu=1}^t h_{\nu,n}F_{\nu,n}=0.
    \end{equation}
    If $t=1$, the normalized sum is identically zero and the
    conclusion follows.

    Suppose that $t\ge2$ and that $h_\nu\not\equiv0$ for some $\nu$.
    Lemma~\ref{lem:stabilization} gives $\ell<p-1$ and, for every
    $\mu\ne\nu$, $k_\nu\not\preceq_{\ell+1}k_\mu$.
    Thus $(F_{\nu,n}/F_{\mu,n})_n$ is not locally uniformly bounded
    on $D(\rho_{\ell+1})$. The divergence alternative in that lemma
    gives a fixed compact set $K_\mu\Subset D(\rho_{\ell+1})$ such that
    \[
        L_{\mu,n}:=\max_{K_\mu}
        \left|\frac{F_{\nu,n}}{F_{\mu,n}}\right|\longrightarrow\infty.
    \]
    Put $B_n=(\min_{\mu\ne\nu}L_{\mu,n})^{1/2}$.
    Then $B_n\to\infty$ and $L_{\mu,n}/B_n\to\infty$ for every
    $\mu\ne\nu$. On $D(\rho_\ell)$ define
    \[
        A_{\mu,n}=B_n\frac{F_{\mu,n}}{F_{\nu,n}},\qquad
        a_{\mu,n}=h_{\mu,n}\frac{F_{\mu,n}}{F_{\nu,n}}
        \quad(\mu\ne\nu).
    \]
    These $A_{\mu,n}$ are units, and
    $a_{\mu,n}/A_{\mu,n}=h_{\mu,n}/B_n\to0$ locally uniformly.
    By \eqref{eq:partition-identity},
    \[
        \sum_{\mu\ne\nu}a_{\mu,n}
        =-h_{\nu,n}\longrightarrow-h_\nu\not\equiv0.
    \]
    Apply Proposition~\ref{thm:absorption} to these $t-1$ terms
    after rescaling $D(\rho_\ell)$ to $\D$. For some $\mu\ne\nu$,
    along a further subsequence,
    \[
        \frac1{A_{\mu,n}}=\frac{F_{\nu,n}}{B_nF_{\mu,n}}
        \longrightarrow0
        \quad\text{locally uniformly on }D(\rho_\ell r_{t-1}).
    \]
    Since $r_{t-1}\ge r_{p-1}=\sigma$, this disk contains
    $D(\rho_{\ell+1})$. This contradicts
    \[
        \max_{K_\mu}\left|\frac1{A_{\mu,n}}\right|
        =\frac{L_{\mu,n}}{B_n}\longrightarrow\infty.
    \]
    Therefore every $h_\nu$ vanishes identically. Each $J_\nu$ is a
    $C$-class on $D(\rho_\ell)$, with dominant index $k_\nu$, and hence
    also on $D(\varepsilon_p)$. Each contains at least two indices.
\end{proof}
	
	The conclusion is invariant under disk automorphisms.
	
	\begin{cor}\label{cor:centers}
		For every $z_0\in\D$, the sequence in Theorem~\ref{thm:main}
		has, after extraction, a partition into $C$-classes on
		\[
		\left\{z\in\D:
		\left|\frac{z-z_0}{1-\overline{z_0}z}\right|<\varepsilon_p\right\}.
		\]
		The partition and subsequence may depend on $z_0$.
	\end{cor}
	
	\begin{proof}
		Apply the theorem to the functions composed with
		$w\mapsto(z_0+w)/(1+\overline{z_0}w)$.
	\end{proof}

	\section{The sharp radius for five functions}\label{sec:sharp-five}
	
	Put $r_*=2-\sqrt3$. We use the hyperbolic distance $d_{\D}$ associated
	with the length element $2|dz|/(1-|z|^2)$. Thus
	\begin{equation}\label{eq:five-hyperbolic-radius}
		\operatorname{diam}_{\D}D(r)=4\operatorname{arctanh}r,
		\qquad 4\operatorname{arctanh}r_*=\log3.
	\end{equation}
	For a disk $D(\rho)$, its intrinsic hyperbolic distance will be denoted by
	$d_{D(\rho)}$.
	
	The harmonic functions used in the example below occur in
	\cite{Ere96b}. The extremal problem stated in \cite{Ere96b} for the
	harmonic functions discussed there was solved by Tamrazov
	\cite{Tam98}; see also the addition of April~24, 1996 in
	\cite{Ere96b}. We shall instead compare harmonic functions
	along the geodesic joining two prescribed points.
	
	\medskip\noindent\textbf{A comparison along a geodesic.}
	
	For real $x\in(-1,1)$ put
	\[
	P_x(\theta)=\frac{1-x^2}{1-2x\cos\theta+x^2}.
	\]
	We first prove an inequality.
	
\begin{lem}\label{lem:two-point-kernel}
    For every $0<q<r_*$ there is $\varepsilon_q>0$ such that,
    for every positive harmonic function $U$ on $\D$ and all real
    $t,s$ with $0\le t\le q$ and $|s|\le t$, we have
    \begin{equation}\label{eq:two-point-kernel}
        k_-(t,s)U(-t)+k_+(t,s)U(t)
        \ge(1+\varepsilon_q)U(s),
    \end{equation}
    where
    \[
        k_-(t,s)=\frac{(1-t)(1-s)}{(1+t)(1+s)},\qquad
        k_+(t,s)=\frac{(1-t)(1+s)}{(1+t)(1-s)}.
    \]
\end{lem}
	
	\begin{proof}
		By the Poisson representation of a positive harmonic function, it
		suffices to compare the kernels. Write $c=\cos\theta$. Since
		$1+r_*^2=4r_*$, we have
		\[
		\frac{1-r_*}{1+r_*}P_{-r_*}(\theta)=\frac1{2+c},
		\qquad
		\frac{1-r_*}{1+r_*}P_{r_*}(\theta)=\frac1{2-c}.
		\]
		A simple computation gives
		\begin{equation}\label{eq:kernel-factorization}
			\begin{split}
				&\frac{1-s}{(1+s)(2+c)}
				+\frac{1+s}{(1-s)(2-c)}
				-\frac{1-s^2}{1-2sc+s^2}\\
				&\qquad=
				\frac{\bigl((1+s^2)c-2s\bigr)^2+12s^2(1-c^2)}
				{(1-s^2)(4-c^2)(1-2sc+s^2)}
				\ge0.
			\end{split}
		\end{equation}
		For fixed $c\in[-1,1]$,
		\[
		\frac{d}{dt}\frac{(1-t)^2}{1\mp2tc+t^2}
		=-\frac{2(1-t^2)(1\mp c)}{(1\mp2tc+t^2)^2}.
		\]
		Both derivatives are nonpositive, and at least one is negative.
		Consequently, replacing $r_*$ in the first two kernels by any
		$t<r_*$ makes the difference in \eqref{eq:kernel-factorization}
		strictly positive. After division by $P_s(\theta)$, this difference
		has a positive minimum on
		\[
		\{(t,s,c):0\le t\le q,\ |s|\le t,\ -1\le c\le1\}.
		\]
		Integration against the positive measure representing $U$ proves
		\eqref{eq:two-point-kernel}.
	\end{proof}
	
	The next formulation makes the role of hyperbolic distance explicit.
	
	\begin{cor}\label{cor:geodesic-comparison}
		Fix $0<d_0<\log3$. There is $\varepsilon>0$, depending only on $d_0$,
		with the following property. Let $U,P,Q$ be nonnegative harmonic
		functions on $\D$, with $U>0$, and let $\alpha,\beta\in\D$ satisfy
		$d_{\D}(\alpha,\beta)\le d_0$. If $C_0\ge0$ and
		\[
		Q(\alpha)\ge U(\alpha)-C_0,
		\qquad
		P(\beta)\ge U(\beta)-C_0,
		\]
		then, on the hyperbolic geodesic segment $\Gamma$ joining $\alpha$ and $\beta$
		we have
		\begin{equation}\label{eq:geodesic-negative}
			U(z)-P(z)-Q(z)\le-\varepsilon U(z)+2C_0.
		\end{equation}
	\end{cor}
	
	\begin{proof}
		An automorphism of $\D$ takes the endpoints of the segment
		to $-t,t$, with
		$4\operatorname{arctanh}t=d_{\D}(\alpha,\beta)$. Put
		$q=\tanh(d_0/4)<r_*$. On the resulting real segment, Harnack's
		inequality gives
		\[
		Q(s)\ge k_-(t,s)Q(-t),
		\qquad
		P(s)\ge k_+(t,s)P(t).
		\]
		Both coefficients belong to $[0,1]$. Adding these inequalities and
		applying Lemma~\ref{lem:two-point-kernel} proves
		\eqref{eq:geodesic-negative}, with $\varepsilon=\varepsilon_q$.
		Composition with the inverse automorphism returns to $\Gamma$.
		Coincident endpoints are included by taking $t=0$.
	\end{proof}
	
	\medskip\noindent\textbf{The estimate for two terms.}
	
	The following is the more precise version of
	Proposition~\ref{thm:absorption} for the case $m=2$.
	
	\begin{prop}\label{prop:sharp-two-absorption}
		Let $A_n,B_n$ be units on $\D$, and let $a_n,b_n\in\OO(\D)$ satisfy
		\begin{equation*}
			a_n/A_n\longrightarrow0,\qquad
			b_n/B_n\longrightarrow0,\qquad
			s_n:=a_n+b_n\longrightarrow s\not\equiv0
		\end{equation*}
		locally uniformly on $\D$. Let $\Omega$ be a nonempty open subset
		of $\D$ with $\operatorname{diam}_{\D}\Omega\le\log3$.
		After passage to a subsequence, either $1/A_n\to0$ or $1/B_n\to0$
		locally uniformly on $\Omega$.
	\end{prop}
	
	\begin{proof}
		Suppose that neither reciprocal tends to $0$ on any subsequence.
		After discarding finitely many terms, there are compact sets
		$K_A,K_B\Subset\Omega$, a constant $C_0\ge0$, and points
		$\alpha_n\in K_A$, $\beta_n\in K_B$ such that
		\begin{equation}\label{eq:two-failure-points}
			u_n(\alpha_n)\le C_0,\qquad v_n(\beta_n)\le C_0,
			\qquad u_n=\log|A_n|,\quad v_n=\log|B_n|.
		\end{equation}
		Put $K=K_A\cup K_B$. Since $K\Subset\Omega$, there is $\eta>0$
		such that the closed hyperbolic $\eta$-neighborhood of $K$ is contained
		in $\Omega$. Hence
		\[
		\operatorname{diam}_{\D}K<\log3.
		\]
		Indeed, if $x,y\in K$ satisfied $d_{\D}(x,y)=\log3$, we could extend
		the hyperbolic geodesic from $x$ through $y$ by a distance smaller than
		$\eta$, obtaining a point of $\Omega$ at distance greater than $\log3$
		from $x$. Choose a neighborhood $V$ of $K$ with compact closure in
		$\Omega$ and
		\[
		\operatorname{diam}_{\D}\overline V<\log3.
		\]
		Surround the finitely many zeros of $s$ in $K$ by pairwise disjoint
		closed disks contained in $V$, with no zero of $s$ on their
		boundaries. If $\alpha_n$ lies inside one of these disks, the maximum
		modulus principle for $1/A_n$ permits us to move it to the boundary
		while preserving \eqref{eq:two-failure-points}. Move $\beta_n$ in
		the same way, using $1/B_n$. The resulting points lie in a fixed
		compact set
		\[
		K'\subset V\setminus\{s=0\}.
		\]
		Thus $|s_n|\ge c_0>0$ on $K'$ for all sufficiently large $n$.
		Since $a_n/A_n\to0$ locally uniformly and
		$|A_n(\alpha_n)|\le e^{C_0}$ by \eqref{eq:two-failure-points},
		we have $a_n(\alpha_n)/s_n(\alpha_n)\to0$. Similarly,
		$b_n(\beta_n)/s_n(\beta_n)\to0$, and therefore
		$a_n(\beta_n)/s_n(\beta_n)=1-b_n(\beta_n)/s_n(\beta_n)\to1$.
		Thus
		\begin{equation}\label{eq:endpoint-change}
			\frac{a_n(\alpha_n)}{s_n(\alpha_n)}\longrightarrow0,
			\qquad
			\frac{a_n(\beta_n)}{s_n(\beta_n)}\longrightarrow1.
		\end{equation}
		
		Choose $r<1$ with $K'\subset D(r)$. On $K'\times K'$, the distances
		$d_{D(\rho)}$ converge uniformly to $d_{\D}$ as $\rho\uparrow1$.
		We can therefore choose
		\[
		r<\rho_0<\rho_1<1,\qquad 0<d_1<\log3,
		\]
		so that $s$ has no zero on the closed annulus
		$\rho_0\le|z|\le\rho_1$ and
		\begin{equation}\label{eq:two-endpoint-distances}
			d_{D(\rho)}(\zeta,\xi)\le d_1
			\qquad(\zeta,\xi\in K',\ \rho_0\le\rho<1).
		\end{equation}
		Here we use monotonicity of hyperbolic distance under inclusion of
		disks. The condition concerning the annulus can be arranged after
		increasing $\rho_0$, since the zeros of $s$ are isolated.
		
		Evaluating $a_n=s_n-b_n$ at $\beta_n$ and
		$b_n=s_n-a_n$ at $\alpha_n$ gives
		\begin{equation}\label{eq:two-interior-norms}
			\norm{a_n}{r}\ge c_0/2,\qquad
			\norm{b_n}{r}\ge c_0/2
		\end{equation}
		eventually. Define
		\[
		M_n(t)=\mean{\max(0,u_n,v_n)}{t},\qquad
		\epsilon_n=\max_{|z|\le\rho_1}
		\max\{|a_n/A_n|,|b_n/B_n|\}.
		\]
		On the chosen annulus, for some fixed $c_1>0$,
		\[
		c_1\le |s_n|
		\le 2\epsilon_n e^{\max(u_n,v_n)}
		\]
		for all large $n$. Hence $\max(u_n,v_n)\to\infty$ uniformly there,
		and $M_n(\rho_0)\to\infty$. Lemma~\ref{lem:growth} provides radii
		\[
		\rho_n\in[\rho_0,\rho_1),\qquad
		M=M_n(\rho_n),\qquad
		\widehat\rho_n=\rho_n+M^{-1}<\rho_1,\qquad
		M_n(\widehat\rho_n)\le2M.
		\]
		In the rest of the proof $M$ depends on $n$ and tends to infinity.
		Eventually $|a_n|\le|A_n|$ and $|b_n|\le|B_n|$ on
		$\overline{D(\rho_1)}$. Lemma~\ref{lem:logderivative}, applied with
		\eqref{eq:two-interior-norms}, therefore yields
		\begin{equation}\label{eq:two-log-errors}
			m(\rho_n,a_n'/a_n)+m(\rho_n,b_n'/b_n)=O(\log M).
		\end{equation}
		
		Let $U_n,H_n$ be the harmonic extensions in $D(\rho_n)$ of the
		boundary functions $\max(u_n,v_n)$ and $\min(u_n,v_n)$, respectively.
		Then
		\[
		U_n>0,\qquad U_n(0)=M,\qquad H_n=u_n+v_n-U_n.
		\]
		Both $P_n=U_n-v_n$ and $Q_n=U_n-u_n$ are nonnegative harmonic
		functions. Let $\Gamma_n$ be the hyperbolic geodesic segment in
		$D(\rho_n)$ joining $\alpha_n$ to $\beta_n$.
		Apply Corollary~\ref{cor:geodesic-comparison} after rescaling the disk,
		using \eqref{eq:two-failure-points} and
		\eqref{eq:two-endpoint-distances}. We obtain a fixed $\varepsilon>0$
		such that
		\begin{equation}\label{eq:path-negative-majorant}
			H_n(z)\le-\varepsilon U_n(z)+2C_0
			\qquad(z\in\Gamma_n).
		\end{equation}
		A centered hyperbolic ball is geodesically convex, so
		$\Gamma_n\subset\overline{D(r)}$. Its hyperbolic length is at most
		$d_1$, and its Euclidean length is at most $\rho_n d_1/2$.
		Harnack's inequality and \eqref{eq:path-negative-majorant} give fixed
		constants $c,C>0$ such that, for all large $n$,
		\[
		cM\le U_n\le CM\quad\hbox{on }\overline{D(r)},
		\qquad H_n\le-cM\quad\hbox{on }\Gamma_n.
		\]
		
		We extend the last estimate to a neighborhood of fixed width.
		Set $V_n=U_n-H_n=P_n+Q_n\ge0$. Uniformly for
		$w\in\overline{D(r)}$ and $\rho_n\ge\rho_0$, Harnack's inequality
		gives a factor $q_\delta\downarrow1$ as $\delta\downarrow0$ such that
		\[
		U_n(z)\le q_\delta U_n(w),\qquad
		V_n(z)\ge q_\delta^{-1}V_n(w)
		\qquad(|z-w|\le\delta).
		\]
		For $w\in\Gamma_n$, it follows that
		\[
		H_n(z)
		\le(q_\delta-q_\delta^{-1})CM-q_\delta^{-1}cM.
		\]
		Fix $0<\delta<\rho_0-r$ sufficiently small. For some $c_2>0$,
		\begin{equation}\label{eq:two-tubular-negativity}
			H_n(z)\le-c_2M
			\qquad\bigl(\operatorname{dist}(z,\Gamma_n)\le\delta\bigr)
		\end{equation}
		for all large $n$, where $\operatorname{dist}$ denotes Euclidean
		distance.
		
		Consider the holomorphic function
		\[
		W_n=s_na_n'-s_n'a_n
		=b_na_n'-a_nb_n'
		=-s_nb_n'+s_n'b_n.
		\]
		The functions $s_n,s_n'$ are uniformly bounded on
		$\overline{D(\rho_1)}$. The first and last expressions for $W_n$
		give, almost everywhere on $|z|=\rho_n$,
		\[
		\log|W_n|\le\min(u_n,v_n)+E_n,
		\qquad
		E_n=C_1+\log^+|a_n'/a_n|+\log^+|b_n'/b_n|\ge0,
		\]
		where $C_1$ is fixed. The singularities of $E_n$ are integrable,
		and its mean is $O(\log M)$ by \eqref{eq:two-log-errors}.
		Poisson comparison, with kernels uniformly bounded on
		$\overline{D(r+\delta)}$, gives
		\[
		\log|W_n(z)|\le H_n(z)+O(\log M)
		\qquad(|z|\le r+\delta).
		\]
		When $W_n\equiv0$, this inequality is understood in the extended
		sense. Together with \eqref{eq:two-tubular-negativity}, it implies
		\begin{equation*}
			\sup_{\operatorname{dist}(z,\Gamma_n)\le\delta}|W_n(z)|
			\longrightarrow0.
		\end{equation*}
		
		Choose $R_1\in(r,r+\delta)$ so that $s$ has no zero on $|z|=R_1$.
		Enclose all its zeros in $D(R_1)$ in pairwise disjoint closed disks
		contained in $D(R_1)$, disjoint from $K'$, with diameters less than
		$\delta$ and no zeros on their boundaries. Replace the portions of
		$\Gamma_n$ inside these disks by boundary arcs. The resulting paths
		$\widetilde\Gamma_n$ have the same endpoints and stay within distance
		$\delta$ of $\Gamma_n$. Their lengths are uniformly bounded: the
		supporting circle or line of the geodesic meets each disk boundary
		in at most two points, so each replacement adds at most one
		circumference. All the paths lie in the fixed compact set obtained
		from $\overline{D(R_1)}$ by removing the interiors of these disks.
		This set has no zero of $s$, and therefore $|s_n|\ge c_3>0$ there
		eventually. Integration of $(a_n/s_n)'=W_n/s_n^2$ now gives
		\[
		\left|
		\frac{a_n(\beta_n)}{s_n(\beta_n)}
		-\frac{a_n(\alpha_n)}{s_n(\alpha_n)}
		\right|
		\le
		\operatorname{length}(\widetilde\Gamma_n)c_3^{-2}
		\sup_{\widetilde\Gamma_n}|W_n|
		\longrightarrow0.
		\]
		This contradicts \eqref{eq:endpoint-change}.
	\end{proof}
	
	The proof uses only two failure points in fixed compact subsets of
	$\Omega$ and harmonic comparison in the ambient disk; the joining
	geodesic need not lie in $\Omega$.
	In particular, Proposition~\ref{prop:sharp-two-absorption} applies
	on $D(r_*)$ by \eqref{eq:five-hyperbolic-radius}. The example below
	will show that neither this radius nor the diameter $\log3$ can be
	increased, even when $s_n=1$.
	
	\medskip\noindent\textbf{The sharp example.}
	
	We refine the construction of \cite[Section~4]{Ere96b}.
	Define
	\begin{equation*}
		\Phi(z)=\frac{2z}{\sqrt{1-z^4}},
		\qquad
		L(z)=\frac{1-z^2}{1+z^2}-\frac{1-z}{2(1+z)},
		\qquad u=\Rea L.
	\end{equation*}
	The square root is the holomorphic branch equal to one at zero,
	which exists because $1-z^4$ has no zeros in $\D$. Put
	$v=-\Rea(\Phi^2)$. A calculation gives
	\begin{equation}\label{eq:five-majorant-identity}
		u-v=\frac12\Rea\frac{1+z}{1-z}>0.
	\end{equation}
	Also $u(z)>0$ for $\Rea z\ge0$. Indeed, if $z=x+iy$ and $s=|z|^2$,
	then
	\[
	u(z)=(1-s)
	\left\{\frac{1+s}{|1+z^2|^2}
	-\frac1{2|1+z|^2}\right\}.
	\]
	Multiplication of the expression inside the braces by
	$2|1+z^2|^2|1+z|^2$ gives
	\[
	(1+s)^2+4x(1+s)+4y^2>0
	\qquad(x\ge0).
	\]
	
	The real parts of $\Phi(z)$ and $z$ have the same sign. To see this,
	use
	\[
	\Ima\frac{w}{1-w^2}
	=\frac{(1+|w|^2)\Ima w}{|1-w^2|^2}
	\]
	with $w=z^2$. In the right half-disk, a zero of $\Rea\Phi$ would force
	$\Phi^2$ to be real and nonpositive. The displayed identity rules
	this out off the real axis, and $\Phi$ is positive on the positive
	real axis. Connectedness gives the sign throughout the right
	half-disk. Oddness treats the left half-disk, and $\Phi$ is purely
	imaginary on the imaginary axis.
	
	For positive integers $n$, let
	\[
	G_n(w)=\frac12+\sqrt{\frac n\pi}
	\int_0^w e^{-n\zeta^2}\dd\zeta.
	\]
	The integral defines an entire function, and
	$G_n(w)+G_n(-w)=1$. For $w=x+iy$ with $x\le0$, shifting the
	integration contour to a horizontal line gives
	\[
	|G_n(w)|
	\le \sqrt{\frac n\pi}e^{ny^2}
	\int_{-\infty}^{x}e^{-nt^2}\dd t
	\le\frac12 e^{-n\Rea(w^2)}.
	\]
	The vertical segment at real part $-T$ contributes zero as
	$T\to\infty$. For the last inequality, write $t=x-\tau$,
	$\tau\ge0$, and use $(x-\tau)^2\ge x^2+\tau^2$.
	Symmetry gives, for $x\ge0$,
	\[
	|G_n(w)|\le1+\frac12e^{-n\Rea(w^2)}.
	\]
	Set
	\[
	a_n=G_n\circ\Phi,\qquad b_n=1-a_n,\qquad
	A_n=2n e^{nL},\qquad B_n(z)=A_n(-z).
	\]
	Since $\Phi$ is odd, $b_n(z)=a_n(-z)$. Its sign property, the
	positivity of $u$ in the right half-disk, and
	\eqref{eq:five-majorant-identity} imply
	\[
	|a_n(z)|\le\frac32e^{nu(z)},\qquad
	|b_n(z)|\le\frac32e^{nu(-z)}.
	\]
	Consequently,
	\begin{equation}\label{eq:five-uniform-ratios}
		\sup_{\D}|a_n/A_n|\le\frac3{4n},\qquad
		\sup_{\D}|b_n/B_n|\le\frac3{4n}.
	\end{equation}
	Thus the five functions
	\begin{equation}\label{eq:five-counterexample}
		A_n,\qquad B_n,\qquad a_n-A_n,\qquad b_n-B_n,\qquad -1
	\end{equation}
	are units with zero sum for every $n\ge1$.
	
	At zero, $A_n(0)=B_n(0)=2ne^{n/2}\to\infty$. On the negative real
	axis,
	\[
	L(-t)=\frac{(1+t)(1-4t+t^2)}{2(1-t)(1+t^2)}
	\qquad(0<t<1).
	\]
	This expression is negative precisely for $t>r_*$. Hence, for every
	$r_*<R\le1$, choosing $r_*<t<R$ gives
	\begin{equation}\label{eq:five-vanishing-points}
		A_n(-t)\longrightarrow0,\qquad B_n(t)\longrightarrow0.
	\end{equation}
	
	We now prove Theorem~\ref{thm:sharp-five} and the optimality assertions
	following it.
	
	\begin{proof}[Proof of Theorem~\ref{thm:sharp-five}]
		Apply Cartan's theorem to the given sequence of five units.
		After extraction, either all indices form a $C$-class, or there are
		two disjoint $C$-classes. In the latter case, if these classes do not
		cover all five indices, both have size two. Relabel and divide all
		functions by $-f_{5,n}$ so that $f_{5,n}=-1$ and the two classes are
		$\{1,3\}$ and $\{2,4\}$, with dominant indices $1$ and $2$, respectively.
		This reduction is the one used in
		\cite{Ere96b}.
		
		Put
		\[
		A_n=f_{1,n},\qquad B_n=f_{2,n},\qquad
		a_n=f_{1,n}+f_{3,n},\qquad b_n=f_{2,n}+f_{4,n}.
		\]
		The definition of $C$-classes gives $a_n/A_n\to0$ and $b_n/B_n\to0$, while
		$a_n+b_n=1$. Proposition~\ref{prop:sharp-two-absorption} applies on
		$\Omega$. Whichever reciprocal tends to zero allows the 
		constant function $f_{5,n}$
		to be added to its corresponding pair: its quotient tends
		to zero, and so does the normalized sum of the enlarged class.
		The two resulting classes form the required partition.
		
		To prove optimality, fix $r_*<R\le1$ and consider
		\eqref{eq:five-counterexample}. The first four functions tend in
		modulus to infinity at zero. By
		\eqref{eq:five-uniform-ratios} and
		\eqref{eq:five-vanishing-points}, each tends to zero at one of the
		points $-t,t$ in $D(R)$. A class containing the constant $-1$ cannot
		have one of these four indices as a dominant index, since the
		quotient of the constant would be unbounded at that point. Nor can
		the index of the constant be a dominant index of a class containing any of the other
		functions, since their quotients are unbounded at zero. The
		singleton consisting of the constant is not a $C$-class. These
		obstructions persist under every extraction, so no subsequence has
		a $C$-class partition on $D(R)$.
		
		By \eqref{eq:five-hyperbolic-radius}, the positive assertion on
		$D(r_*)$ and these examples give $R_5=2-\sqrt3$. For any proposed
		universal diameter larger than $\log3$, choose $R>r_*$ sufficiently
		close to $r_*$ that $4\operatorname{arctanh}R$ is smaller than that
		diameter. The same example on $\Omega=D(R)$ proves that $\log3$
		is optimal.
	\end{proof}
	
	\section{Zero directions of the Kobayashi--Royden pseudometric}
	\label{sec:geometry}
	
	We apply Theorem~\ref{thm:main} to holomorphic functions on the unit
	disk obtained by pulling back the monomials of a Laurent polynomial.
	The point is that a map from $\D$ to $(\C^*)^N$ has zero-free
	coordinate functions, so both positive and negative powers are
	holomorphic on $\D$.
	
	\begin{proof}[Proof of Theorem~\ref{thm:torus-zero}]
		Let $Y\subset(\C^*)^N$ be as in Theorem~\ref{thm:torus-zero}.
		If $\gamma_{x,\bm{v}}(\C)\subset Y$, the discs
		$z\mapsto\gamma_{x,\bm{v}}(nz)$ show that $\kappa_Y(x;\bm{v})=0$.
		Conversely, suppose that $\kappa_Y(x;\bm{v})=0$. There is nothing to prove
		when $\bm{v}=0$. Otherwise choose $\bm F_n\in\OO(\D,Y)$ such that
		\[
		\bm F_n(0)=x,\qquad \bm F_n'(0)=t_n\bm{v},\qquad t_n\to\infty.
		\]
		Write $\bm F_n=(F_{1,n},\ldots,F_{N,n})$.
		
		Choose finitely many nonzero Laurent polynomials defining $Y$ in the
		torus; for $Y=(\C^*)^N$ there are no equations to check. Consider one
		of them, written with distinct exponent vectors as
		\[
		P(\bm{z})=\sum_{\bm{\alpha}\in E}c_{\bm{\alpha}}\bm{z}^{\bm{\alpha}},
		\qquad c_{\bm{\alpha}}\ne0,\qquad E\subset\mathbb Z^N,
		\]
		where
		\[
		\bm{z}^{\bm{\alpha}}=\prod_{j=1}^Nz_j^{\alpha_j},\qquad
		\bm F_n^{\bm{\alpha}}:=\prod_{j=1}^NF_{j,n}^{\alpha_j}.
		\]
		Every $F_{j,n}$ is a unit on $\D$, so $\bm F_n^{\bm{\alpha}}$ is a
		unit even when some $\alpha_j$ are negative. Since
		$P(\bm F_n(z))=0$ for every $z\in\D$, the functions
		\[
		f_{\bm{\alpha},n}:=c_{\bm{\alpha}}\bm F_n^{\bm{\alpha}},
		\qquad \bm{\alpha}\in E,
		\]
		are units on $\D$ and their sum is identically zero.
		If $|E|=1$, this is impossible because a single nonzero monomial is a
		unit. If $|E|=2$, the quotient of the two terms is identically $-1$,
		so the two indices already form a $C$-class. Thus only the case
		$|E|\ge3$ requires Theorem~\ref{thm:main}.
		
		Let $I\subset E$ be one of the resulting $C$-classes, and let
		$\bm{\beta}\in I$ be a dominant index. Evaluating its normalized sum at
		$0$ and using $\bm F_n(0)=x$ gives
		\[
		\sum_{\bm{\alpha}\in I}c_{\bm{\alpha}}x^{\bm{\alpha}}=0.
		\]
		The quotients $f_{\bm{\alpha},n}/f_{\bm{\beta},n}$ are locally bounded
		on a fixed disk about zero. Cauchy's estimate applied to these quotients
		at $0$ gives
		\[
		t_n\frac{c_{\bm{\alpha}}x^{\bm{\alpha}}}
		{c_{\bm{\beta}}x^{\bm{\beta}}}
		(\bm{\alpha}-\bm{\beta})\cdot\bm{\lambda}=O(1),
		\qquad \bm{\alpha}\in I,
		\]
		where $\bm{\alpha}\cdot\bm{\lambda}=\sum_{j=1}^N\alpha_j\lambda_j$ and
		$\lambda_j=v_j/x_j$. Since $t_n\to\infty$, it follows that
		$(\bm{\alpha}-\bm{\beta})\cdot\bm{\lambda}=0$ for every
		$\bm{\alpha}\in I$. Hence $\bm{\alpha}\cdot\bm{\lambda}$ has a common
		value $\nu_I$ on $I$, and
		\[
		P(\gamma_{x,\bm{v}}(\zeta))
		=\sum_I e^{\nu_I\zeta}
		\sum_{\bm{\alpha}\in I}c_{\bm{\alpha}}x^{\bm{\alpha}}=0.
		\]
		Repeating this argument for the finitely many defining polynomials
		proves that $\gamma_{x,\bm{v}}$ takes its values in $Y$.
	\end{proof}

	The same argument also controls sequences of tangent vectors with
	varying base points. Write
	\[
	\mathcal N_Y=\{(x,\bm{v})\in TY:\kappa_Y(x;\bm{v})=0\}.
	\]
	
	\begin{prop}\label{prop:torus-null-set}
		For every smooth closed algebraic subvariety $Y\subset(\C^*)^N$,
		the set $\mathcal N_Y$ is a closed algebraic subset of $TY$. For
		every compact set $K\subset TY\setminus\mathcal N_Y$,
		\[
		\inf_{(x,\bm{v})\in K}\kappa_Y(x;\bm{v})>0.
		\]
	\end{prop}
	
	\begin{proof}
		Consider a defining Laurent polynomial
		$P(\bm{z})=\sum_{\bm{\alpha}\in E}c_{\bm{\alpha}} \bm{z}^{\bm{\alpha}}$, and put $q=|E|$.
		For $\lambda_j=v_j/x_j$, the identity
		$P\circ\gamma_{x,\bm{v}}=0$ is equivalent to the finite set of equations
		\[
		\sum_{\bm{\alpha}\in E}c_{\bm{\alpha}} x^{\bm{\alpha}}(\bm{\alpha}\cdot\bm{\lambda})^k=0,
		\qquad 0\le k<q.
		\]
		Indeed, group the terms by the distinct values of
		$\bm{\alpha}\cdot\bm{\lambda}$. If there are $d$ such values, the first $d$
		equations form a Vandermonde system for the grouped coefficients.
		They force every coefficient to vanish, which proves the identity;
		the reverse implication follows by differentiation at zero.
		These are regular algebraic equations in $(x,\bm{v})$ on the tangent
		bundle of the torus. Apply them to each of the finitely many defining
		polynomials and use Theorem~\ref{thm:torus-zero}. This proves the
		assertion about $\mathcal N_Y$.
		
		Suppose, contrary to the second assertion of the proposition, that
		for some compact set $K\subset TY\setminus\mathcal N_Y$ one has
		$\inf_K\kappa_Y=0$. Then there are $(x_n,\bm{v}_n)\in K$ converging,
		after passage to a subsequence, to $(x,\bm{v})\in K$, and discs
		$\bm F_n\in\OO(\D,Y)$ with
		\[
		\bm F_n(0)=x_n,\qquad \bm F_n'(0)=t_n\bm{v}_n,\qquad t_n\to\infty.
		\]
		For each defining polynomial apply Theorem~\ref{thm:main} to
		$c_{\bm{\alpha}}\bm F_n^{\bm{\alpha}}$, with the cases of one or two
		terms treated as above, and extract successively for the finitely many
		polynomials. In each $C$-class $I$, choose a dominant index
		$\bm{\beta}$. Then
		evaluation of the normalized sum gives
		\[
		\sum_{\bm{\alpha}\in I}c_{\bm{\alpha}} x^{\bm{\alpha}}=0.
		\]
		Put $\lambda_{j,n}=v_{j,n}/x_{j,n}$, so that
		$\bm{\lambda}_n\to\bm{\lambda}=(v_j/x_j)_{j=1}^N$. Bounded quotient
		derivatives give
		\[
		t_n\frac{c_{\bm{\alpha}} x_n^{\bm{\alpha}}}{c_{\bm{\beta}} x_n^{\bm{\beta}}}
		(\bm{\alpha}-\bm{\beta})\cdot\bm{\lambda}_n=O(1),
		\qquad \bm{\alpha}\in I.
		\]
		The fraction converges to a nonzero number. Dividing by $t_n$ and
		passing to the limit therefore yields
		$(\bm{\alpha}-\bm{\beta})\cdot\bm{\lambda}=0$. The proof of
		Theorem~\ref{thm:torus-zero} now gives
		$\gamma_{x,\bm{v}}(\C)\subset Y$, contrary to
		$(x,\bm{v})\notin\mathcal N_Y$.
	\end{proof}
	
	We next specialize to projective space with hyperplanes omitted.
	Following the symmetric coordinates of Kiernan and Kobayashi
	\cite[pp.~208--209]{KK73}, write
	\[
	H_p=\left\{[x_1:\cdots:x_p]\in\mathbb P^{p-1}:
	\sum_{j=1}^p x_j=0\right\},
	\qquad
	X_p=H_p\setminus\bigcup_{j=1}^p\{x_j=0\},
	\qquad p\ge3.
	\]
	Thus $H_p\simeq\mathbb P^{p-2}$, and the omitted hyperplanes are in
	general position. A tuple of units with zero sum defines a
	holomorphic map to $X_p$. Conversely, the global affine chart
	$x_p=1$ represents every holomorphic map to $X_p$ by such a tuple.
	
	For a nonempty subset $I\subset\{1,\ldots,p\}$, define
	\[
	\pi_I:X_p\longrightarrow\mathbb P^{|I|-1},\qquad
	\pi_I([x])=[x_j]_{j\in I},
	\qquad
	H_I=\left\{[w_j]_{j\in I}:\sum_{j\in I}w_j=0\right\}.
	\]
	When $|I|=1$, the set $H_I$ is empty. For a partition
	$\mathcal P=\{I_1,\ldots,I_q\}$ of $\{1,\ldots,p\}$, put
	\[
	\Pi_{\mathcal P}=(\pi_{I_1},\ldots,\pi_{I_q}),
	\qquad Y_{\mathcal P}=\prod_{\nu=1}^qH_{I_\nu}.
	\]
	
	\begin{prop}\label{prop:projective-equivalence}
		Fix $p\ge3$ and $0<R\le1$. The $C$-class partition property on
		$D(R)$ is equivalent to the following assertion: every sequence of
		holomorphic maps $F_n:\D\to X_p$ has a subsequence and a partition
		$\mathcal P$ such that $\Pi_{\mathcal P}\circ F_n$ converges locally
		uniformly on $D(R)$ to a holomorphic map into $Y_{\mathcal P}$.
	\end{prop}
	
	\begin{proof}
		Represent $F_n$ by a tuple $(f_{1,n},\ldots,f_{p,n})$ of units with
		zero sum. In a $C$-class $I$ with dominant index $k$, the boundedness
		on compact subsets
		implies, after extraction of a subsequence,
		\[
		f_{j,n}/f_{k,n}\longrightarrow g_j\quad(j\in I),
		\qquad g_k=1,
		\qquad \sum_{j\in I}g_j=0.
		\]
		Hence $\pi_I\circ F_n\to[g_j]_{j\in I}\in H_I$. A common
		subsequence works for the finitely many $C$-classes.
		
		Conversely, suppose that $\pi_I\circ F_n$ converges to a holomorphic
		map $G_I:D(R)\to H_I$. Each approximating map avoids the coordinate
		hyperplanes. Hurwitz's theorem implies that $G_I$ either avoids each
		such hyperplane or is contained in it
		\cite[Proposition~2, pp.~204--205]{KK73}. Choose $k\in I$ for which the
		image of $G_I$ is not contained in $\{w_k=0\}$. It therefore avoids
		that hyperplane. In this affine chart the quotients
		$f_{j,n}/f_{k,n}$ converge locally uniformly, and their sum tends to
		zero. Thus $I$ is a $C$-class.
	\end{proof}
	
	At a point $x=[x_1:\cdots:x_p]\in X_p$, call a partition
	$\mathcal P=\{I_1,\ldots,I_q\}$ admissible if
	$\sum_{j\in I_\nu}x_j=0$ for every $\nu$. This condition is
	independent of the choice of homogeneous coordinates. Denote the
	finite set of admissible partitions by $\mathfrak P(x)$. Each such
	partition defines a torus
	\[
	\mathcal T_{\mathcal P,x}
	=\left\{[t_{\nu(j)}x_j]_{j=1}^p:
	t_1,\ldots,t_q\in\C^*\right\}\subset X_p,
	\]
	where $\nu(j)$ is the unique index such that $j\in I_{\nu(j)}$. Simultaneous scaling of
	all $t_\nu$ does not change the projective point, so this torus is
	isomorphic to $(\C^*)^{q-1}$. The one-part partition gives the
	point $x$.
	
	\begin{cor}\label{cor:projective-zero-directions}
		For every $x\in X_p$,
		\[
		\{\bm{v}\in T_xX_p:\kappa_{X_p}(x;\bm{v})=0\}
		=\bigcup_{\mathcal P\in\mathfrak P(x)}
		T_x\mathcal T_{\mathcal P,x}.
		\]
	\end{cor}
	
	\begin{proof}
		Use the chart $x_p=1$, in which $X_p$ is the smooth closed
		subvariety of $(\C^*)^{p-1}$ defined by
		$x_1+\cdots+x_{p-1}+1=0$. Represent $\bm{v}$ by
		$(v_1,\ldots,v_{p-1},0)$ and put
		$\lambda_j=v_j/x_j$ for $j<p$ and $\lambda_p=0$.
		By Theorem~\ref{thm:torus-zero}, the equality
		$\kappa_{X_p}(x;\bm{v})=0$ holds precisely when
		\[
		\sum_{j=1}^p x_j e^{\lambda_j\zeta}=0
		\qquad(\zeta\in\C).
		\]
		Group the indices according to the distinct values of $\lambda_j$.
		The linear independence of exponentials with distinct exponents
		shows that the coordinate sum over each group is zero.
		These groups form an admissible partition, and the rates
		$\lambda_j$ are constant on each part. This is exactly the
		condition that $\bm{v}$ belong to the tangent space of the corresponding
		torus. Conversely, a tangent vector to an admissible torus has rates
		constant on each part of the corresponding partition, and the same displayed identity follows
		from their coordinate sums. Theorem~\ref{thm:torus-zero} completes
		the proof.
	\end{proof}
	
	Every part of an admissible partition contains at least two indices. Thus the zero
	directions at $x$ form a finite union of complex linear subspaces,
	each of dimension at most $\lfloor p/2\rfloor-1$.
	Corollary~\ref{cor:projective-zero-directions} describes the zero
	directions, including those over the diagonal hyperplanes
	$\sum_{j\in I}x_j=0$, $2\le|I|\le p-2$. This description supplements the positivity of the Kobayashi--Royden pseudometric
	away from their union established in \cite[Theorem~3 and Theorem~6]{KK73}.
	
	\begingroup
	\raggedright
	
	\endgroup

\begin{thebibliography}{Ere96b}
		
\bibitem[Ahl79]{Ahl79}
L.~V.~Ahlfors, \emph{Complex Analysis}, 3rd ed.,
McGraw--Hill, New York, 1979.
		\bibitem[Ber06]{Ber06}
		W.~Bergweiler, \emph{Bloch's principle}, Comput. Methods Funct. Theory
		\textbf{6} (2006), no.~1, 77--108.
		\doi{10.1007/BF03321119}.
		
		\bibitem[Blo26]{Blo26}
		A.~Bloch, \emph{Sur les syst\`emes de fonctions holomorphes \`a
			vari\'et\'es lin\'eaires lacunaires}, Ann. Sci. \'Ecole Norm. Sup. (3)
		\textbf{43} (1926), 309--362.
		\doi{10.24033/asens.772}.
		
		\bibitem[Blo26a]{Blo26a}
		A.~Bloch, \emph{La conception actuelle de la th\'eorie des fonctions enti\`eres
			et m\'eromorphes}, L'Enseign. Math. \textbf{25} (1926), 83--103.
		
		\bibitem[Bor97]{Bor97}
		\'E.~Borel, \emph{Sur les z\'eros des fonctions enti\`eres},
		Acta Math. \textbf{20} (1897), 357--396.
		
		\bibitem[BD10]{BD10}
		A.~Bostan and P.~Dumas, \emph{Wronskians and linear independence},
		Amer. Math. Monthly \textbf{117} (2010), no.~8, 722--727.
		\doi{10.4169/000298910X515785}.
		
		\bibitem[Car28]{Car28}
		H.~Cartan, \emph{Sur les syst\`emes de fonctions holomorphes \`a
			vari\'et\'es lin\'eaires lacunaires et leurs applications}, Ann. Sci.
		\'Ecole Norm. Sup. (3) \textbf{45} (1928), 255--346.
		\doi{10.24033/asens.786}.

		
		
		\bibitem[Ere96a]{Ere96a}
		A.~Eremenko, \emph{A counterexample to Cartan's conjecture on holomorphic
			curves omitting hyperplanes}, Proc. Amer. Math. Soc. \textbf{124} (1996),
		no.~10, 3097--3100.
		\doi{10.1090/S0002-9939-96-03392-8}.
		
		\bibitem[Ere96b]{Ere96b}
		A.~Eremenko, \emph{Holomorphic curves omitting five planes in projective
			space}, Amer. J. Math. \textbf{118} (1996), no.~6, 1141--1151.
		\doi{10.1353/ajm.1996.0048}.
		Author's version with the addendum of April~24, 1996:
		\url{https://www.math.purdue.edu/~eremenko/dvi/cartan5.pdf}.
		
\bibitem[Ere15]{Ere15}
		A.~Eremenko, \emph{Modified Cartan's conjecture}, problem note, April~4, 2015.
		\url{https://www.math.purdue.edu/~eremenko/dvi/cartconj.pdf}.
		
\bibitem[KP26]{KP26}
S.~N.~Karp and K.~Purbhoo, \emph{Universal Pl\"ucker coordinates for
    the Wronski map and positivity in real Schubert calculus},
J. Amer. Math. Soc. (2026), in press.
\doi{10.1090/jams/1087}.

		
		\bibitem[KK73]{KK73}
		P.~Kiernan and S.~Kobayashi, \emph{Holomorphic mappings into projective
			space with lacunary hyperplanes}, Nagoya Math. J. \textbf{50} (1973),
		199--216.
		\doi{10.1017/S0027763000015646}.
		
		\bibitem[Lan87]{Lan87}
		S.~Lang, \emph{Introduction to complex hyperbolic spaces}, Springer-Verlag,
		New York, 1987.
		\doi{10.1007/978-1-4757-1945-1}.
		
		\bibitem[Mil92]{Mil92}
		J.~Miles, \emph{A sharp form of the lemma on the logarithmic derivative},
		J. London Math. Soc. (2) \textbf{45} (1992), no.~2, 243--254.
		\doi{10.1112/jlms/s2-45.2.243}.
		
		
	\bibitem[Mon27]{Mon27}
P.~Montel, \emph{Le\c{c}ons sur les Familles normales
    de fonctions analytiques et leurs applications},
Paris, Gauthier-Villars, 1927.

        \bibitem[Pic79]{Pic79}
        E.~Picard, \emph{Sur une propri\'et\'e des fonctions enti\`eres},
        C. R. Acad. Sci. Paris \textbf{88} (1879), 1024--1027.
		
		\bibitem[Roy71]{Roy71}
		H.~L.~Royden, \emph{Remarks on the Kobayashi metric}, in
		\emph{Several complex variables, II} (Proc. Internat. Conf., Univ.
		Maryland, College Park, Md., 1970), Lecture Notes in Math., vol.~185,
		Springer-Verlag, Berlin, 1971, pp.~125--137.
		\doi{10.1007/BFb0058768}.
		
		\bibitem[Tam98]{Tam98}
		P.~M.~Tamrazov, \emph{Eremenko's extremal problem on harmonic functions},
		Complex Variables Theory Appl. \textbf{35} (1998), no.~4, 319--336.
		\doi{10.1080/17476939808815090}.
		
		\bibitem[Yam26]{Yam26}
		K.~Yamanoi, \emph{Bloch's principle for holomorphic maps into
			subvarieties of semi-abelian varieties}, Publ. Res. Inst. Math. Sci.
		\textbf{62} (2026), no.~3, 455--628.
		\doi{10.4171/PRIMS/62-3-1}.

        \bibitem[Zha26]{Zha26}
T.~Zhang,
\emph{A Lean~4 formalization of the modified Cartan conjecture},
Lean~4 formalization with mathlib, Lean~v4.34.0-rc1,
commit \texttt{c91afc2d8c92dc1394a244cfdf1faaac92380b7f}, 2026.
Available at
\url{https://github.com/zhangteng2000/modified-cartan-lean}.
		
	\end{thebibliography}
\end{document}